\documentclass[reqno]{amsart}

\usepackage{amsmath}
\usepackage{graphicx}
\usepackage{amsmath,amssymb,amsfonts,amsthm,graphicx,amsopn}
\usepackage{mathrsfs,enumitem}
\PassOptionsToPackage{gray}{xcolor}
\usepackage{tikz-cd}
\usepackage{tikz}
\usepackage{stmaryrd}
\usepackage{mathdots, verbatim}
\usepackage{array,float}
\usepackage{fullpage}
\usepackage[utf8]{inputenc}
\usepackage{multirow}

\DeclareMathOperator{\T}{T}

\DeclareMathOperator{\rk}{rk}
\DeclareMathOperator{\Aut}{Aut}
\DeclareMathOperator{\id}{id}

\DeclareMathOperator{\Pic}{Pic}

\DeclareMathOperator{\NS}{NS}

\DeclareMathOperator{\MW}{MW}
\DeclareMathOperator{\TMW}{TMW}

\DeclareMathOperator{\Tors}{Tors}

\DeclareMathOperator{\Or}{O}

\usepackage[style=alphabetic,maxbibnames=99]{biblatex}
 
\title{K3 Surfaces of zero entropy admitting an elliptic fibration with only irreducible fibers}
\author{Giacomo Mezzedimi}
\address{Institut f\"ur Algebraische Geometrie, Leibniz Universit\"at Hannover, Welfengarten 1, 30167 Hannover, Germany.}
\email{mezzedimi@math.uni-hannover.de}
\date{\today}

\begin{document}

\newcommand{\der}{\partial}
\newcommand{\derc}{\overline{\partial}}
\newcommand{\Q}{\mathbb{Q}}
\newcommand{\R}{\mathbb{R}}
\newcommand{\Z}{\mathbb{Z}}
\newcommand{\C}{\mathbb{C}}
\newcommand{\Hyp}{\mathbb{H}}
\newcommand{\N}{\mathbb{N}}
\newcommand{\A}{\mathbb{A}}
\newcommand{\F}{\mathbb{F}}
\newcommand{\Prj}{\mathbb{P}}
\newcommand{\Aff}{\mathbb{A}}
\newcommand{\eps}{\varepsilon}
\newcommand{\Sc}{\mathscr{S}}
\newcommand{\Ac}{\mathscr{A}}
\newcommand{\Lc}{\mathcal{L}}
\newcommand{\Rc}{\mathcal{R}}
\newcommand{\Fc}{\mathscr{F}}
\newcommand{\Qc}{\mathscr{Q}}
\newcommand{\Ec}{\mathscr{E}}
\newcommand{\Gc}{\mathscr{G}}
\newcommand{\Jcc}{\mathscr{J}}
\newcommand{\Uc}{\mathcal{U}}
\newcommand{\Mvar}{\mathcal{M}}
\newcommand{\Cc}{\mathcal{C}}
\newcommand{\Ccc}{\mathscr{C}}
\newcommand{\Hc}{\mathscr{H}}
\newcommand{\Kc}{\mathcal{K}}
\newcommand{\Ic}{\mathcal{I}}
\newcommand{\Jc}{\mathcal{J}}
\newcommand{\Icc}{\mathscr{I}}
\newcommand{\mf}{\mathfrak{m}}
\newcommand{\Ol}{\mathcal{O}}
\newcommand{\El}{\mathcal{E}}
\newcommand{\smallvdots}{\vphantom{\int\limits^x}\smash{\vdots}}
\newcommand{\smallddots}{\vphantom{\int\limits^x}\smash{\ddots}}
\newcommand*{\longhookrightarrow}{\ensuremath{\lhook\joinrel\relbar\joinrel\rightarrow}}
\newcommand*\circled[1]{\tikz[baseline=(char.base)]{
            \node[shape=circle,draw,solid,inner sep=2pt] (char) {#1};}}
\newcommand{\rsim}{\rotatebox[origin=c]{-90}{$\sim$}}
\newcommand{\sdots}{\reflectbox{\ddots}}
\newcommand{\Mod}[1]{\ (\mathrm{mod}\ #1)}

\newtheorem{teo}{Theorem}[section]
\newtheorem{prop}[teo]{Proposition}
\newtheorem{lemma}[teo]{Lemma}
\newtheorem*{lemma*}{Lemma}
\newtheorem{cor}[teo]{Corollary}
\newtheorem*{esercizio}{Esercizio}
\newtheorem{prob}[teo]{Problem}
\theoremstyle{definition}
\newtheorem*{ack}{\textbf{Ackowledgments}}
\newtheorem*{conv}{\textbf{Conventions}}
\newtheorem{defn}[teo]{Definition}
\newtheorem{alg}[teo]{Algorithm}
\theoremstyle{remark}
\newtheorem{oss}[teo]{Remark}
\newtheorem{question}[teo]{Question}
\newtheorem{ossi}[teo]{Remarks}
\newtheorem{esempio}[teo]{Example}
\newtheorem{esempi}[teo]{Examples}

\maketitle

\begin{abstract}
We classify complex K3 surfaces of zero entropy admitting an elliptic fibration with only irreducible fibers. These surfaces are characterized by the fact that they admit a unique elliptic fibration with infinite automorphism group. We furnish an explicit list of $32$ N\'eron-Severi lattices corresponding to such surfaces. Incidentally, we are able to decide which of these $32$ classes of surfaces admit a unique elliptic pencil. Finally, we prove that all K3 surfaces with Picard rank $\ge 19$ and infinite automorphism group have positive entropy. 
\end{abstract}

\section*{Introduction}

Let $X$ be a smooth projective K3 surface over an algebraically closed field. The study of the group $\Aut(X)$ of automorphisms of $X$ is a central topic at the intersection of algebraic, arithmetic and differential geometry. Since the early works by Nikulin \cite{nikulin10}, Kond\=o \cite{kondo} and Vinberg \cite{vinberg}, many have tried to understand explicitly the structure of the group $\Aut(X)$ using very different approaches. A very successful approach in the last 20 years has been via complex dynamics and entropy, pioneered by Cantat \cite{cantat3} and McMullen \cite{mcmullen1}. Our aim is to combine this with the huge lattice-theoretical machinery classically used to study K3 surfaces.\newline

The first step towards the understanding of the group $\Aut(X)$ was made by Nikulin \cite{nikulin10}, Vinberg \cite{vinberg07} and Kond\=o \cite{kondo}: Nikulin classified the N\'eron-Severi lattices of complex K3 surfaces with a finite automorphism group for Picard rank $\rho(X)\ne 4$, and Vinberg completed the classification in Picard rank $4$. Their work relies on the theory of lattices developed by Nikulin in the 70's. However, when the automorphism group becomes infinite, very little is known. For example, we can describe the full automorphism group only of some K3 surfaces (see Vinberg's examples \cite{vinberg} or Shimada's recent algorithm \cite{shimada1}).

Our goal is to identify a class of complex K3 surfaces with an infinite but simple automorphism group. Let $C\subseteq X$ be an elliptic curve. Then we can consider the subgroup $\Aut(X,|C|)<\Aut(X)$ of automorphisms of $X$ preserving the elliptic pencil $|C|$. $\Aut(X,|C|)$ can be related to the group of automorphisms of the generic member $C_\eta$ of the elliptic pencil $|C|$, and is in general very well understood. Hence a natural approach is to relate $\Aut(X)$ to the groups $\Aut(X,|C|)$, with $C$ varying among the elliptic curves on $X$. This can be rephrased in terms of \textit{entropy} of automorphisms: a result by Cantat \cite{cantat2} shows that the ``most regular'' automorphisms (said of \textit{zero entropy}) of a K3 surface are either the periodic ones or those preserving some elliptic pencil. Hence $\Aut(X)$ can be understood from the groups $\Aut(X,|C|)$ when all automorphisms of $X$ have zero entropy. A K3 surface with this property is said \textit{of zero entropy}, otherwise $X$ is said to have \textit{positive entropy}, and it is believed that K3 surfaces with positive entropy have a much more complicated automorphism group.

A more precise description than the one above is the following:

\begin{teo}[\cite{oguiso1}] \label{teo:introfac}
A smooth complex projective K3 surface with $|\Aut(X)|=\infty$ has zero entropy if and only if $\Aut(X)$ coincides with $\Aut(X,|C|)$ for a certain elliptic curve $C\subseteq X$.
\end{teo}

A corollary of Theorem \ref{teo:introfac} states that $X$ has zero entropy if and only if it admits a unique elliptic pencil $|C|$ with $|\Aut(X,|C|)|=\infty$. The advantage of this characterization is that it is purely lattice-theoretical, as the uniqueness can be read off the N\'eron-Severi lattice of $X$. A classification of complex K3 surfaces admitting a unique elliptic pencil with infinite automorphism group was asked for by Nikulin in \cite{nikulin1}.\newline

We address this classification problem in the case when $X$ satisfies a technical condition. More precisely, we want $X$ to admit an elliptic fibration (i.e. an elliptic pencil with a section) with only irreducible fibers, i.e. with only nodal or cuspidal singular fibers. This forces the Picard rank $\rho(X)\ge 3$, since every K3 surface admitting an elliptic fibration and with Picard rank at most $2$ has a finite automorphism group. Notice that not all K3 surfaces with $\rho(X)\ge 3$ satisfy this technical condition, but there are only finitely many N\'eron-Severi lattices of K3 surfaces not satisfying it.

Our main result is the following:

\begin{teo} \label{teo:intro}
Let $X$ be a smooth complex projective K3 surface with $|\Aut(X)|=\infty$. Suppose that $X$ admits an elliptic fibration with only irreducible fibers. Then $X$ has zero entropy if and only if the N\'eron-Severi lattice $\NS(X)$ belongs to an explicit list of $32$ lattices.
\end{teo}

The reader can find this list in Theorem \ref{thm:list0}. Incidentally, we also classify which of these $32$ classes of K3 surfaces admit other elliptic pencils during the proof of Theorem \ref{teo:intro}.

The classification in Theorem \ref{teo:intro} is obtained in three steps; we are going to outline the main ideas of each of them. By above we can consider $\rho(X)\ge 3$. If $X$ is a K3 surface admitting an elliptic fibration, then the sublattice of $\NS(X)$ generated by the elliptic curve $F$ and its zero section $S_0$ induces an orthogonal decomposition $\NS(X)=U\oplus L$. $U$ is a hyperbolic plane generated by $F$ and $S_0$, and $L=U^\bot$ is an even, negative definite lattice that describes the structure of the elliptic fibration, for example its reducible fibers and the group of its sections. When $\rho(X)=3$, the rank of $L$ is $1$, hence the intersection form on $\NS(X)$ is completely governed by a unique number, which coincides with the determinant of $\NS(X)$. Since automorphisms preserve the nef cone, we can rephrase our problem in terms of the nef cone of such surfaces. We then show that the structure of the nef cone can be understood by solving some congruences involving the determinant of $\NS(X)$. This allows us to show that $X$ has zero entropy if and only if $\det(\NS(X))$ satisfies a certain arithmetic property (cf. Theorem \ref{teo:finale3}).

When $\rho(X)\ge 4$, the intersection form on $\NS(X)$ depends on a lattice of rank $\rho(X)-2\ge 2$, hence it is impracticable to generalize the previous approach. However, a classical tool of lattice theory comes to our help. Consider again the orthogonal decomposition $\NS(X)=U\oplus L$ introduced above. The \textit{genus} of $L$ is a finite set of even, negative definite lattices parametrizing the structure of possible elliptic fibrations on $X$. If the genus of $L$ contains only $L$, we say that the genus is trivial, and this implies that all elliptic fibrations on $X$ are isomorphic (i.e. they all have the same structure).

An old result by Watson \cite{watson1}, recently corrected by Lorch and Kirschmer \cite{nebe}, completely classifies even, negative definite lattices with a trivial genus, and furnishes us with an explicit list. Our second step towards the classification in Theorem \ref{teo:intro} consists in proving that if $X$ satisfies the assumptions of Theorem \ref{teo:intro} and has zero entropy, then its N\'eron-Severi lattice must decompose as $\NS(X)=U\oplus L$, with $L$ having trivial genus. A priori it could happen that a K3 surface has many elliptic fibrations, but a unique one with infinite automorphism group. We rule this out by proving that if a K3 surface admits one elliptic fibration with only irreducible fibers and another elliptic fibration with finite automorphism group, then it admits a third intermediate elliptic fibration with infinite automorphism group. We obtain this result by studying the genera of root lattices.

Our third and final step amounts to studying the lattices in Watson's list. This list is infinite, as it contains all the multiples of some lattices. Using a recursive argument and the classification in Picard rank $3$ obtained previously we bound the determinant of the N\'eron-Severi lattice of a K3 surface of zero entropy. This allows us to restrict to a finite number of cases, and the classification is then completed by checking individually these remaining lattices.\newline

It is natural to ask what happens if we remove the technical condition in Theorem \ref{teo:intro}. If $\rho(X)=20$ is maximal, the K3 surface is said \textit{singular}, and in this case Oguiso \cite{oguiso1} has proven that $X$ always has positive entropy. Using the techniques introduced above, we are able to generalize his result to Picard rank $19$:

\begin{teo} \label{teo:introsing}
All smooth complex projective K3 surfaces with Picard rank $\ge 19$ and infinite automorphism group have positive entropy.
\end{teo}

As previously recalled, we have a complete list of N\'eron-Severi lattices of K3 surfaces with finite automorphism group (cf. \cite{nikulin10}). A quick inspection of such list shows that there exists a unique such N\'eron-Severi lattice of rank $\ge 19$, which is $U\oplus E_8\oplus E_8\oplus A_1$. This, combined with Theorem \ref{teo:introsing}, shows that any K3 surface $X$ of Picard rank $\ge 19$ and $\NS(X)\not\cong U\oplus E_8\oplus E_8 \oplus A_1$ has positive entropy. \newline

The outline of the paper follows closely the previous discussion. In Section 1 we give an overview on automorphisms on K3 surfaces and the basics of lattice theory. In Section 2 we recall the definition of entropy, the classification of automorphisms of K3 surfaces due to \cite{cantat2}, and Oguiso's Theorem \ref{teo:introfac}. In Section 3 we lay the groundwork to prove the main result. More precisely, we use Nikulin's theory of lattices to find sufficient conditions for a K3 surface to have positive entropy. In Section 4,5,6 we explain the three steps discussed above, in order to obtain the classification in Theorem \ref{teo:intro}. Finally, in Section 7 we prove Theorem \ref{teo:introsing}.

\begin{conv}
Throughout the paper we will always work over $\C$. We have used the software \texttt{Magma} to implement all the algorithms. The interested reader can find these algorithms and some computational data in the ancillary folder on ArXiv (\textit{ArXiv:1912.08583}).
\end{conv}

\begin{ack}
First af all, I want to thank my advisor Matthias Schütt for suggesting the problem and supervising the progress of this paper. I am grateful to Serge Cantat, Simon Brandhorst, Alberto Cattaneo and Mauro Fortuna for the many useful discussions, and to Edgar Ayala for carefully reading this manuscript. I thank Keiji Oguiso for pointing out important references, and Viacheslav Nikulin for finding a mistake in an earlier draft. I also thank the anonymous referee for the the useful comments. Finally, I am indebted to Victor Lozovanu for helping me improve the structure of this paper.
\end{ack}

\section{Setup}

\subsection{Automorphisms of projective smooth K3 surfaces}

Let $X$ be a smooth projective K3 surface over $\C$. The reference for this section is \cite{huybrechts1}. $H^2(X,\Z)$ is naturally endowed with a unimodular intersection pairing, making it isomorphic to the \textit{K3 lattice}
$$\Lambda_{K3}=U^{3}\oplus E_8^{ 2},$$
where $U$ is the hyperbolic plane and $E_8$ is the unique (up to isometry) even unimodular negative definite lattice of rank $8$. In particular the signature of $H^2(X,\Z)$ is $(3,19)$. Since the canonical bundle of $X$ is trivial, there exists a unique (up to scalars) nowhere-vanishing $(2,0)$-form $\omega_X$ on $X$.

The Néron-Severi group $\NS(X)=\Pic(X)$ is a hyperbolic sublattice of $H^2(X,\Z)$, i.e. it has signature $(1,\rho(X)-1)$, where $\rho(X)=\rk{\NS(X)}$ is the \textit{Picard rank} of $X$. The \textit{transcendental lattice} $\T(X)=\NS(X)^\bot\subset H^2(X,\Z)$ is the orthogonal complement of $\NS(X)$ in $H^2(X,\Z)$, and its complexification $\T(X)_\C=\T(X)\otimes \C$ contains the $(2,0)$-form $\omega_X$.

A peculiarity of K3 surfaces is that we can study their group of automorphisms as a subgroup of the group of isometries $\Or(H^2(X,\Z))$. Indeed, any automorphism $f\in \Aut(X)$ acts naturally as an isometry $f^*$ on $H^2(X,\Z)\cong \Lambda_{K3}$, and the map 
$$\Aut(X)\longrightarrow \Or(H^2(X,\Z))$$
sending $f$ to $f^*$ turns out to be injective. Similarly we have a map
$$\Aut(X)\longrightarrow \Or^+(\NS(X)),$$
where $\Or^+(\NS(X))$ is the group of isometries of $\NS(X)$ preserving the positive cone $\mathcal{C}_X$. There exists a chamber decomposition of the positive cone $\mathcal{C}_X$, and the Weyl subgroup $W<\Or^+(\NS(X))$ acts transitively on the set of chambers. Recall that $W$ is the subgroup generated by reflections across smooth $(-2)$-curves, i.e. generated by the reflections $s_\delta:C\mapsto C+(C\cdot \delta)\delta$ for all $\delta\in \NS(X)$ corresponding to smooth rational curves. The transitivity above can be rephrased in geometric terms: if $\alpha\in \mathcal{C}_X$ is an element in the positive cone, either it is nef, or there exists a smooth rational curve $\delta$ such that $\alpha\cdot \delta<0$ (cf. \cite{huybrechts1}, Corollary 8.1.7). Then $\alpha':=s_\delta(\alpha)$ has positive intersection with $\delta$, and we can repeat the process with $\alpha'$. After a \textit{finite} number of reflections, the element $\alpha$ becomes nef. Equivalently, there exists a unique nef element in the orbit $W\alpha$.

Therefore the quotient $\Or^+(\NS(X))/W$ can be viewed as the subgroup of $\Or^+(\NS(X))$ preserving the nef cone. Since automorphisms also preserve the nef cone, there is a strict interplay between these two groups, which is made explicit by the following theorem. If $L$ is an even lattice, we denote by $A_L=L^\vee/L$ the \textit{discriminant group} of the lattice, endowed with the induced quadratic form $q_L$ with values in $\Q/2\Z$. Moreover, we denote by $\Or_{\Delta^+}(\NS(X))$ the subgroup of $\Or^+(\NS(X))$ of isometries preserving the set of smooth curves $\Delta^+=\{ C \mid C \cong \Prj^1 \}\subseteq \NS(X)$.

\begin{prop}[\cite{PSS}, Section 7 - \cite{huybrechts1}, Chapter 15] \label{prop:tutto}
Let $X$ be a smooth complex projective K3 surface. Then:
\begin{enumerate}
\item The homomorphism
$$\Aut(X) \longrightarrow \Or^+(\NS(X))/W$$
has finite kernel and cokernel.
\item The group $\Aut(X)$ is isomorphic to the group
$$\{(\alpha,\beta) \in \Or_{\Delta^+}(\NS(X))\times \Or(\T(X)) \mid \overline{\alpha}=\overline{\beta}\in \Or(A_{\NS(X)})=\Or(A_{\T(X)})\},$$
where $\overline{\alpha}$ and $\overline{\beta}$ are the induced isometries of the isometric discriminant groups $A_{\NS(X)}=A_{\T(X)}$.
\end{enumerate}
\end{prop}

Another classical result that we are going to need is the following:

\begin{teo}[\cite{huybrechts1}, Corollary 3.3.4, 3.3.5] \label{teo:huy2}
Let $X$ be a smooth projective K3 surface over $\C$, and $f\in \Aut(X)$. Then there exists an $n \in \N$ such that $(f^*)^n=\id$ on $\T(X)$. If moreover the Picard rank of $X$ is odd, then $f^*=\pm \id$ on $\T(X)$.
\end{teo}

We now recall some basic facts about elliptic fibrations on K3 surfaces.

\begin{defn}
Let $X$ be a smooth projective surface. A \textit{genus $1$ fibration} on $X$ is a proper and flat morphism $\pi: X \rightarrow C$ to a smooth projective curve $C$ such that its generic fiber is smooth of genus $1$. If $\pi$ admits a section, we will say that $\pi$ is an \textit{elliptic fibration}, since every smooth fiber inherits the structure of an elliptic curve.
\end{defn}

When $X$ is a K3 surface, genus $1$ fibrations are in bijection with primitive nef elements $0\ne F\in \NS(X)$ with $F^2=0$. Indeed, it is easy to check that the linear system $|F|$ associated to any such $F$ induces a genus $1$ fibration $|F|:X\rightarrow \Prj^1$. 

If $|F|$ is an elliptic fibration on $X$, i.e. there exists an irreducible $(-2)$-curve $S$ with $FS=1$, then we will denote by $$\MW(F)=\{S\in \NS(X) \text{ irreducible $(-2)$-curve with $FS=1$}\}$$
the \textit{Mordell-Weil} group of the fibration. It has a natural group structure, induced by the group structure on the generic fiber of the fibration (after the choice of an $S_0\in \MW(F)$). Clearly $\rk(\MW(F))\le \rho(X)-2$, and equality can fail depending on the singular fibers of the elliptic fibration $|F|$ (cf. the Shioda-Tate formula, \cite{shioda2}, Corollary 1.5). Equality holds if and only if $|F|$ has only irreducible fibers (i.e. only nodal or cuspidal singular fibers), and in this case we will say that $|F|$ has \textit{maximal rank}.

There exists an embedding
$$\MW(F) \longhookrightarrow \Aut(X)$$
sending a section $S$ to the associated translation $\tau_S$. More precisely, if $S\in \MW(F)$ corresponds to a point $x$ in the generic fiber $F_\eta$, the translation by $x$ induces an automorphism of the genus $1$ curve $F_\eta$, which can be extended to an automorphism $\tau_S$ of the whole $X$. Notice that $\tau_S$ acts on a smooth fiber $F_0$ simply by the translation by $S \cap F_0$.

We can refine the previous embedding by underlining that $\MW(F)$ embeds into the subgroup $\Aut(X,|F|)<\Aut(X)$ of automorphisms $f$ of $X$ preserving the fibration $|F|$, i.e. such that $f^*F=F\in \NS(X)$. The group $\Aut(X,|F|)$ contains the group of automorphisms of the generic fiber $F_\eta$. $\Aut(F_\eta)$ is in turn generated by $\MW(F)$ (via the usual correspondence between rational points of $F_\eta$ and sections of $F$) and by the finitely many automorphisms respecting the origin of the group law on $F_\eta$. The quotient $\Aut(X,|F|)/\Aut(F_\eta)$, which corresponds to the induced action of $\Aut(X,|F|)$ on the base curve $\Prj^1$, is finite, since it has to permute the critical values of the fibration $|F|$, and it is easy to show that $|F|$ contains at least $3$ singular curves. Indeed, if $F_i$ are the singular fibers, $e_i$ is the Euler characteristic of $F_i$ and $r_i$ is the number of irreducible components in $F_i$ not meeting the zero section, we have $\sum_i{r_i}\le \rk(F^\bot/\langle F\rangle)\le 18$, $\sum_i{e_i}=e(X)=24$, and $e_i-r_i\in \{0,1,2\}$ (see for instance \cite{miranda1}, Lemma IV.3.2), hence there are at least $3$ singular fibers.

\subsection{Lattices and dense sphere packings}

In this section we will recall the basics of dense sphere packings, for ease of reference. We normally have to deal with negative definite lattices, but in order to be coherent with the huge literature, we will also deal with positive definite lattices. Our main reference is \cite{CS}.\newline

If $L$ is any even definite lattice of rank $r$, we denote
$$\min(L)=\min\{ |\|v\|_L| : v \in \Z^r \backslash\{0\} \}.$$

A \textit{root} of $L$ is a vector of norm ($\pm$)$2$. The \textit{root part} of $L$ is the sublattice $L_{root}$ of $L$ generated by its roots. $M$ is an \textit{overlattice} of $L$ if $L$ is a sublattice of $M$ of finite index.

\begin{defn}
A \textit{root lattice} is a lattice that coincides with its root part. A \textit{root-overlattice} is a lattice that is an overlattice of its root part, i.e. $\rk{L}=\rk{L_{root}}$. For the sake of simplicity, we consider the root lattices as a special case of root-overlattices.
\end{defn}

The root lattices are simply the lattices that can be obtained as a direct sum of the lattices $A_n,D_n,E_n$. We can easily compute their discriminant groups:

\begin{center}
\begin{tabular}{ c|c } 
$L$ & $A_L$\\
 \hline
$A_n$ &$\Z/(n+1)\Z$\\
$D_{2n}$ &$\Z/2\Z \times \Z/2\Z$\\
$D_{2n+1} $ &$\Z/4\Z$\\
$E_6$ &$\Z/3\Z$\\
$E_7$ &$\Z/2\Z$\\
$E_8$ &$\{0\}$\\

\end{tabular}
\end{center}

\begin{oss}
The maximal determinant of a root-overlattice of rank $r$ is $2^r$. Indeed, since root-overlattices are overlattices of root lattices, the maximal determinant must be attained at a root lattice (recall that if $M$ is an overlattice of $L$, then $\det(M)=\det(L)/[M:L]^2$). It is immediate to notice that the maximal determinant corresponds to the lattice $A_1^r$, of determinant $2^r$.
\end{oss}

We can also ask the converse question: if $L$ has no roots, can we bound its determinant $\det(L)$ from below? The following is one of the main theorems of dense sphere packings:

\begin{teo}[\cite{CS}, Table 1.2] \label{teo:deltar}
Let $L$ be an even positive definite lattice of rank $r$, with $m=\min(L)$. Then there exists a number $\delta_r >0$ such that
$$\frac{(m/4)^{(r/2)}}{\sqrt{\det(L)}}\le \delta_r.$$
In other words, there exists a constant $\Delta_{m,r}$ depending on $r$ and $m=\min(L)$ such that $\det(L)\ge \Delta_{m,r}$.
\end{teo}

If $L$ has no roots, then $\min(L)\ge 4$. Table 1.2 in \cite{CS} provides some lower bounds for $\det(L)$:

\begin{table}[H]
    \centering

\begin{tabular}{c c c}

\begin{tabular}{ c|c } 
$r=\rk{L}$ & $\Delta_r$\\
 \hline
$1$ &$4$\\
$2$ &$12$\\
$3$ &$32$\\
$4$ &$64$\\
$5$ &$128$\\
$6$ &$192$

\end{tabular}

\begin{tabular}{ c|c } 
$r=\rk{L}$ & $\Delta_r$\\
 \hline
$7$ &$256$\\
$8$ &$256$\\
$9$ &$278$\\
$10$ &$283$\\
$11$ &$266$\\
$12$ &$233$

\end{tabular}

\begin{tabular}{ c|c } 
$r=\rk{L}$ & $\Delta_r$\\
 \hline
$13$ &$191$\\
$14$ &$146$\\
$15$ &$106$\\
$16$ &$73$\\
$17$ &$47$\\
$18$ &$29$

\end{tabular}
\end{tabular}
\caption{Lower bounds for the determinant of definite even lattices without roots.}
\label{table:bound}
\end{table}

Notice that $\Delta_r > 2^r$ if $r\le 7$, and $\Delta_r=2^r$ if $r=8$. 

\begin{defn}
Two even, positive (or negative) definite lattices $L,L'$ of the same rank are \textit{in the same genus} (or equivalently $L'$ is \textit{in the genus of $L$}) if their discriminant groups are isometric (i.e. the two quadratic forms $q_L,q_{L'}$ on $A_L \cong A_{L'}$ are isomorphic). The \textit{genus} of $L$ is the set of lattices in the genus of $L$ up to isometry.
\end{defn}

\begin{prop} \label{prop:rootgenus}
If $r\le 8$, then any lattice $L$ of rank $r$ in the genus of a root-overlattice has $\min(L)=2$. Equivalently, if $r \le 8$, there are no root-overlattices in the genus of a lattice with no roots.
\end{prop}
\begin{proof}
This is obvious if $r\le 7$ as noticed above, since two lattices in the same genus have the same determinant. If $r=8$, we only have to check the claim for $A_1^8$, which is the unique root-overlattice of rank $8$ with determinant $2^8=256$. An easy check with \texttt{Magma} reveals however that $A_1^8$ is unique in its genus (up to isometry). Alternatively, one can use the Siegel mass formula \cite{CS88} and deduce that $A_1^8$ is unique in its genus from the equality
$$m(q)=\frac{1}{10321920}=\frac{1}{|\Or(A_1^8)|},$$
where $q$ is the quadratic form on the discriminant group of $A_1^8$.
\end{proof}

\begin{oss}
The previous result is in general not true if $r>8$. For instance, there exists a lattice with minimum $4$ in the genus of $A_1^{12}$. We will see other similar examples in the following, arising from more geometric constructions.
\end{oss}

\section{Entropy on K3 surfaces}

In this section we are going to recall first the main results on the entropy of automorphisms on complex K3 surfaces, and then we will prove a general criterion to decide whether a K3 surface has zero entropy.\newline

Let $X$ be a complex projective K3 surface. The cohomology group $H^{1,1}(X,\R)$ is a vector space of dimension $20$, endowed with a hyperbolic nondegenerate metric $q_X$. Hence the sheet
$$\mathbb{H}_X=\{c \in H^{1,1}(X,\R) \mid q_X(c)=1\}^+$$
intersecting the K\"{a}hler cone of $X$ is a model for the hyperbolic space $\mathbb{H}^{19}$. Since the automorphism group $\Aut(X)$ of the surface acts as an isometry on $H^2(X,\R)$ and preserves the K\"{a}hler cone, we have a natural map
$$\Aut(X) \longrightarrow \Or(\mathbb{H}_X).$$
Moreover $\Aut(X)$ can be seen as a discrete subgroup of isometries of $H^2(X,\R)$, since it embeds into the isometry group $\Or(H^2(X,\Z))$ of the lattice $H^2(X,\Z)\subseteq H^2(X,\R)$.\newline

The standard theory of hyperbolic geometry classifies isometries of $\mathbb{H}_X$ into three types: \textit{elliptic}, \textit{parabolic} or \textit{hyperbolic}. Recall that $\phi \in \Or(\mathbb{H}_X)$ is
\begin{itemize}
    \item \textit{elliptic}, if $\phi$ fixes an inner point $x\in \mathbb{H}_X \backslash \partial \mathbb{H}_X$;
    \item \textit{parabolic}, if $\phi$ is not elliptic and fixes a unique point in the boundary $\partial \mathbb{H}_X$;
    \item \textit{hyperbolic}, if $\phi$ fixes two points in the boundary $\partial \mathbb{H}_X$.
\end{itemize}

The next theorem interprets the geometric behaviour of automorphisms of $X$ with respect to this classification:

\begin{teo}[\cite{cantat2}] \label{teo:cantat}
Let $f\in Aut(X)$, and denote by $f^*\in \Or(\mathbb{H}_X)$ the induced isometry on the hyperbolic space $\mathbb{H}_X$.
\begin{itemize}
    \item $f^*$ is elliptic if and only if $f$ is periodic (i.e. it has finite order).
    \item $f^*$ is parabolic if and only if $f$ is not periodic and it respects a genus $1$ fibration on $X$ (i.e. there exists a primitive, nef element $F\in \NS(X)$ with $F^2=0$ such that $f^*F=F$). In this case, all eigenvalues of $\phi^*$ have norm $1$.
    \item $f^*$ is hyperbolic otherwise. In this case there exists a Salem number $\lambda>1$ such that $\{\lambda,\frac{1}{\lambda}\}$ is the list of eigenvalues of $f^*$ with norm different from $1$. 
\end{itemize}
\end{teo}

The concept of \textit{entropy} of automorphisms on K3 surfaces is closely related to this classification. The entropy can be defined in much more generality, but we restrict ourselves to the case when $Y$ is a complex projective variety and $g$ an automorphism of $Y$:

\begin{defn}
Let $Y$ be a complex projective variety and $g\in \Aut(Y)$. The \textit{entropy} of $g$ is defined as the quantity $h(g)=\log{\lambda(g^*)}$, where $\lambda(g^*)$ is the spectral radius of the pull-back map $g^*:H^*(Y,\C)\rightarrow H^*(Y,\C)$ on singular cohomology, i.e. the maximum norm of its eigenvalues.
\end{defn}

\begin{oss}
In this case there exists an equivalent, more topological, definition of the entropy, measuring how fast the iterates of $g$ create distinct orbits. See \cite{cantat1} for a nice introduction, and \cite{gromov1}, \cite{yomdin1} for the equivalence of the two definitions. If instead the variety is defined on a field of positive characteristic, there exists a similar definition of the entropy that uses \'etale cohomology; the interested reader can consult \cite{esnault1}.
\end{oss}

If $X$ is a K3 surface as above and $f\in \Aut(X)$, then the pull-back $f^*$ acts as the identity on $H^0(X,\C)\oplus H^4(X,\C)$. Moreover Theorem \ref{teo:huy2} shows that $f^*$ acts with finite order on the complexification $\T(X)_\C$ of the trascendental lattice, so the entropy of $f$ coincides with $\log{\lambda(f^*|_{\NS(X)_\C})}$, where $f^*|_{\NS(X)_\C}$ is the restriction of the pull-back to $\NS(X)_\C\subseteq H^2(X,\C)$. Hence Theorem \ref{teo:cantat} implies immediately:

\begin{cor} \label{cor:zero}
Let $f\in \Aut(X)$ be an automorphism of the K3 surface $X$. Then $f$ has zero entropy if and only if $f^*$ is either elliptic or parabolic. In other words, $h(f)=0$ if and only if $f$ is either periodic or it respects a genus $1$ fibration on $X$.
\end{cor}

Recall that, if $C$ is any elliptic curve on $X$, we have defined the subgroup $\Aut(X,|C|)<\Aut(X)$ of automorphisms of $X$ preserving the genus $1$ fibration induced by $C$ (or equivalently, preserving the elliptic pencil $|C|$). We will call $\Aut(X,|C|)$ the \textit{automorphism group of the fibration $|C|$}. The previous result then shows that an automorphism $f\in \Aut(X)$ of infinite order has zero entropy if and only if $f\in \Aut(X,|C|)$ for some elliptic curve $C\subseteq X$.\\
The groups $\Aut(X,|C|)$ are quite easy to study: for instance, if $C$ induces an elliptic fibration on $X$, then $\Aut(X,|C|)$ coincides up to finite groups with the group $\MW(C)$ of sections of $|C|$, which in turn coincides up to torsion with some $\Z^s$. 

\begin{oss}
If $|F|$ is an elliptic fibration on $X$ and $S \in \MW(F)$ is a section, then the automorphism $\tau_S \in \Aut(X)$ induced by the section $S$ is elliptic (resp. parabolic) if and only if $S$ has finite (resp. infinite) order in $\MW(F)$. 
\end{oss}

We are finally able to indroduce the main protagonist of our paper:

\begin{defn}
A K3 surface $X$ is said to have \textit{zero entropy} (and we write $h(X)=0$) if all of its automorphisms have zero entropy. Otherwise $X$ is said to have \textit{positive entropy}, and we write $h(X)>0$.
\end{defn}

If $X$ has a finite automorphism group, then every $f\in \Aut(X)$ is elliptic, hence $X$ has zero entropy. K3 surfaces with a finite automorphism group have been widely studied by Nikulin (see for instance \cite{nikulin3}, \cite{nikulin4}, \cite{nikulin2}, \cite{nikulin5}, \cite{nikulin6}, \cite{nikulin7}, \cite{nikulin8}); we have in fact a complete classification of Néron-Severi lattices of complex K3 surfaces with a finite automorphism group (see also \cite{kondo} for a description of such automorphism groups). Therefore we are interested in studying K3 surfaces with an \textit{infinite} automorphism group.

If $X$ is a K3 surface with zero entropy, then studying its automorphism group amounts to studying the groups $\Aut(X,|C|)$, with $C$ varying among the elliptic curves $C\subseteq X$. Therefore automorphism groups of K3 surfaces of zero entropy are in some sense the ``easiest'' to understand.\newline

Our goal is to completely classify complex K3 surfaces with an infinite automorphism group and zero entropy. Notice that such surfaces must admit a genus $1$ fibration $|C|$ with infinite $\Aut(X,|C|)$. For, since $\Aut(X)$ is infinite and finitely generated (cf. \cite{sterk}), it must admit an element of infinite order, which has zero entropy by assumption.

\begin{oss} \label{oss:entropyneron}
Having zero entropy really depends only on the N\'eron-Severi lattice $\NS(X)$: indeed $\Aut(X)$ coincides with $\Or^+(\NS(X))/W$ up to finite groups, and an isometry in $\Or^+(\NS(X))$ has zero entropy if and only if one of its powers has zero entropy.
\end{oss}

The following is the main characterization of K3 surfaces of zero entropy:

\begin{teo}[\cite{oguiso1}, Theorem 1.4] \label{thm:og}
Let $X$ be a smooth, projective K3 surface with an infinite automorphism group. Then $X$ has zero entropy if and only if $\Aut(X)=\Aut(X,|C|)$ for a certain elliptic curve $C\subseteq X$. Equivalently, $X$ has zero entropy if and only if there exists a unique genus $1$ fibration $|C|$ on $X$ with infinite automorphism group.
\end{teo}

\begin{cor}
Let $X$ be a smooth, projective K3 surface admitting an elliptic fibration $|F|$ with infinitely many sections. Then $X$ has zero entropy if and only if $|F|$ is the unique elliptic fibration on $X$ with infinitely many sections. 
\end{cor}

\begin{oss}
Nikulin proves in \cite{nikulin1} (cf. Theorem 10) that if $\rho(X)\ge 3$ and $X$ admits at least $2$ genus $1$ fibrations with infinite automorphism group, then it must admit an infinite number of such fibrations.
\end{oss}

In this paper we will consider the case when $X$ admits an elliptic fibration with infinitely many sections; under this assumption we must have $\rho(X)\ge 3$, as it is easy to check that, if $\rho(X) \le 2$, $X$ admits either no elliptic fibrations at all, or every elliptic fibration on $X$ has trivial Mordell-Weil group.

\section{Elliptic K3 surfaces of zero entropy}

This section is the core of the paper: here we prove the first generalities about K3 surfaces of zero entropy, and we introduce the methods that will be used later on to classify such surfaces.\newline

From now on $X$ will always be a smooth complex projective elliptic K3 surface, with Picard rank $\rho(X)\ge 3$. For the moment we assume that $X$ admits an elliptic fibration $|F|$ of maximal rank, i.e. with only irreducible fibers. If $S_0$ is the zero section of the fibration, the unimodularity of the trivial lattice $\langle F,S_0\rangle \cong U$ induces an orthogonal decomposition
$$\NS(X) = \langle F,S_0 \rangle \oplus L,$$
where $L$ is an even negative definite lattice of rank $r=\rho(X)-2$. Notice that $L$ has no roots, because the elliptic fibration $|F|$ has only irreducible fibers.

We will denote in the following by $[x,y,z]\in \NS(X)$ the divisor written with respect to the basis $\{F,S_0,\mathcal{B}\}$ of $\NS(X)$, where $\mathcal{B}$ is a basis of $L$, fixed once and for all.

\begin{oss}
Let $0 \ne D=[x,y,z]\in \NS(X)$ be a divisor (not necessarily irreducible nor reduced) such that $D^2\ge -2$. By Riemann-Roch one of $D$ and $-D$ is effective, and since $FD=y$, we have that $D$ is effective if $y> 0$, while $-D$ is effective if $y<0$. This leads to the following useful characterization:
\end{oss}

\begin{lemma}
Let $X$ be a K3 surface admitting an elliptic fibration with only irreducible fibers. Let $A\in \NS(X)$ be a divisor with $A^2\ge 0$ and $FA\ge 0$. Then $A$ is nef if and only if, for all divisors $D=[x,y,z]\in \NS(X)$ with $D^2=-2$ and $y>0$, the inequality $AD\ge 0$ holds.
\end{lemma}
\begin{proof}
The divisor $A$ is nef if and only if it has non-negative intersection with all smooth rational curves on $X$ (cf. \cite{huybrechts1}, Corollary 8.1.4, 8.1.7). Suppose that there exists $D=[x,y,z]\in \NS(X)$ with $D^2=-2$ and $y>0$ such that $AD<0$. Then by the above remark $D$ is effective, and, since $D^2=-2$, it is forced to split into the sum of some irreducible $(-2)$-curves. The inequality $AD<0$ implies that there exists an irreducible $(-2)$-curve $C$ (which is a summand of $D$) such that $AC<0$, contradicting the nefness of $A$.\\
Conversely, if $A$ is not nef, there exists an effective $(-2)$-curve $C=[x,y,z]$ such that $AC<0$. As above $y \ge 0$, but $y=0$ only if $C$ is contained in a fiber of the elliptic fibration, hence $y>0$ since by assumption there are no reducible fibers.
\end{proof}

\begin{lemma} \label{lemma:int}
Let $X$ be a K3 surface admitting an elliptic fibration $|F|$ with only irreducible fibers. Let $F\ne E=[\alpha,\beta,\gamma],C=[x,y,z]\in \NS(X)$ be effective, primitive divisors such that $E^2=0$ and $C^2=-2$. Then the equation $EC=m$ can be equivalently written as
$$-\frac{1}{2}\|v\|_L=\beta(\beta+my),$$
where $v=y\gamma  - \beta z$. In particular $E$ is nef if and only if 
$$-\frac{1}{2}\|v\|_L - \beta^2 \ge 0$$
for any such $C$, and $E$ induces a genus $1$ fibration with only irreducible fibers if and only if
$$-\frac{1}{2}\|v\|_L - \beta^2 > 0$$
for any such $C$.
\end{lemma}
\begin{proof}
This is a straightforward computation. The self-intersections of $E,C$ force
\begin{equation} \label{eq:div}
\begin{cases}
\alpha= \beta + \frac{-\frac{1}{2}\|\gamma\|_L}{\beta}\\
x= y + \frac{-\frac{1}{2}\|z\|_L-1}{y}
\end{cases}.
\end{equation}
Notice that $\beta \ne 0$ since $E\ne F$, and $y \ne 0$ since we are assuming that $L$ has no roots. Substituting these expressions into the equation
$$m=EC=\alpha y + \beta x -2\beta y + \langle \gamma,z\rangle_L,$$
we obtain easily the desired equation.
\end{proof}

\begin{oss}
If $E\in \NS(X)$ is such that $E^2=0$ and $EF=1$, then it cannot be nef. Indeed, the divisor $E-F$ has self-intersection $-2$ and $(E-F)F=1$, hence $E-F$ is effective and $(E-F)E=-1$.
\end{oss}

The next proposition highlights a certain ``periodicity'' of elliptic curves on $X$; notice that this highly depends on the assumption that $X$ admits an elliptic fibration with only irreducible fibers.

\begin{prop} \label{prop:periodic}
Let $X$ be a K3 surface admitting an elliptic fibration with only irreducible fibers. Let $E=[\alpha,\beta,\gamma], E'=[\alpha',\beta,\gamma']\in \NS(X)$ be primitive elements with $E^2=E'^2=0$ and $\gamma'\equiv \gamma \pmod{\beta}$ (i.e. all the entries are congruent modulo $\beta$). Then $E$ induces a genus $1$ fibration (resp. an elliptic fibration) on $X$ if and only if $E'$ does so.
\end{prop}
\begin{proof}
Say $\gamma=[\gamma_1,\ldots,\gamma_r]$ and assume $\gamma'=[\gamma_1+\beta,\gamma_2,\ldots,\gamma_r]$; $\alpha'$ is an integer by equation (\ref{eq:div}), since $\gamma',\gamma$ are congruent modulo $\beta$. Then $E$ is nef if and only if $E'$ is nef: indeed, if there exists an effective $C=[x,y,z]$ such that $C^2=-2$ and  $EC=m<0$, by Lemma \ref{lemma:int} we have
$$-\frac{1}{2}\|v\|_L=\beta(\beta+my),$$
where $v=y\gamma-\beta z$. If we put $z'=[z_1+y,z_2,\ldots,z_r]$, clearly $v=y\gamma'-\beta z'$ doesn't change, so $E'C'=m<0$, where $C'=[x',y,z']$. Analogously, $E$ has a section if and only if $E'$ has one. We conclude repeating the same argument for each coordinate of $\gamma$.
\end{proof}

The following theorem will be one of our main tools to prove that many K3 surfaces have positive entropy. First, let $X$ be an elliptic K3 surface with $\NS(X)=U\oplus L$. For all primitive sublattices $L'$ of $L$ there exists an elliptic K3 surface $X'$ with N\'eron-Severi lattice $\NS(X')=U\oplus L'$. This follows from the surjectivity of the period map for K3 surfaces (cf. \cite{todorov}, Theorem 1), since $U\oplus L' \hookrightarrow U\oplus L \hookrightarrow \Lambda_{K3}$ embeds primitively into the K3 lattice. The goal of the following theorem is to relate the entropy of $X$ to the entropy of $X'$. For simplicity we will say that $\NS(X)$ has positive entropy if $X$ has positive entropy (since having positive entropy only depends on the N\'eron-Severi lattice).

\begin{teo} \label{teo:extension}
Let $X$ be a K3 surface admitting an elliptic fibration with only irreducible fibers, $\NS(X)=U\oplus L$. Assume further that there exists a primitive sublattice $L'$ of $L$ of corank $1$ such that $U\oplus L'$ has positive entropy. Then $X$ has positive entropy if one of the following conditions holds:
\begin{itemize}
    \item $|\det(L)| > 2|\det(L')|$
    \item $|\det(L)| = 2|\det(L')|$ and $\rho(X) \le 10$.
\end{itemize}
\end{teo}
\begin{proof}
Let us fix a basis for $L$ and consider $L$ as a matrix. We can write
$$L=\left(\begin{array}{c|c}
L' &M\\
\hline
{^t}M &-2k
\end{array}\right).$$
We denote by $[x,y,z,w]\in \NS(X)$ the divisor with coordinates $x,y$ wrt $U$, $z$ wrt $L'$ and $w$ wrt to $\langle -2k \rangle$. By assumption there exists a primitive, effective divisor $E=[\alpha,\beta,\gamma,0]\in \NS(X)$ with $E^2=0$ such that $EC\ge 0$ for all effective $(-2)$-curves $C=[x,y,z,0]\in \NS(X)$. We want to show that $E$ is actually nef. By Lemma \ref{lemma:int} we have that the intersection of $E$ with an effective $(-2)$-curve $C=[x,y,z,w]$ is (up to a positive constant)
$$-\frac{1}{2}{^t}(y\gamma  - \beta z)L'(y\gamma  -\beta z)+ {^t}M(y\gamma -\beta z)\cdot (\beta w) + k\beta^2w^2-\beta^2. $$
We already know that this number is (strictly) positive if $w=0$, so fix any $w\ne 0$. By standard theory of quadratic forms we know that the minimum of the previous expression is attained at $L'v=\beta w M$, and this minimum is in fact
$$\frac{1}{2}\beta^2 w^2 {^t}ML'^{-1}M+k\beta^2w^2 -\beta^2.$$
After dividing by $\beta^2 w^2$, we want to show that
\begin{equation} \label{eq:schur}
\frac{1}{2} {^t}ML'^{-1}M+k -\frac{1}{w^2}\ge 0.  
\end{equation}
It then suffices to prove this inequality for $w^2=1$. Now consider the matrix
$$P=\left(\begin{array}{c|c}
L' &M\\
\hline
{^t}M &-2k+2
\end{array}\right).$$
$L'$ is negative definite, so $P$ is negative semidefinite if and only if $\det(P)$ is opposite in sign to $\det(L')$ (or $0$). But
$$\det(P)=\det(L)+2\det(L'),$$
and $\det(L)$ is in fact opposite in sign to $\det(L')$, hence by the inequality $|\det(L)|\ge 2|\det(L')|$ also $\det(P)$ is opposite in sign to $\det(L')$, i.e. $P$ is negative semidefinite. From the theory of Schur complement this implies that
$$(-2k+2)-{^t}ML'^{-1}M \le 0,$$
which is the desired inequality (\ref{eq:schur}) in the case $w^2=1$.
$E$ induces an elliptic fibration on $X$, since it already had a section on $U\oplus L'$. However, if $|\det(L)| > 2|\det(L')|$, then as above we can conclude that $P$ is actually negative \textit{definite}, so the minimum of the expression in equation (\ref{eq:schur}) is strictly positive, and $E$ has maximal rank. If instead $|\det(L)| = 2|\det(L')|$, $E$ just induces a primitive embedding $i:U \hookrightarrow U\oplus L$. The unimodularity of $i(U)\cong U$ gives an isomorphism $i(U)\oplus i(U)^\bot\cong U\oplus L$, hence $i(U)^\bot$ is in the genus of $L$. Consequently by Proposition \ref{prop:rootgenus} $i(U)^\bot$ cannot be a root-overlattice, since $\rho(X) \le 10$, so $E$ induces an elliptic fibration with infinitely many sections.
\end{proof}

\begin{oss}
This criterion is unfortunately not sharp, as we will see later (see for instance Proposition \ref{prop:counterex}). However this is sufficient to allow us to work with a finite number of lattices: indeed, we will show in Algorithm \ref{alg:lista} that the condition $|\det(L)|>2|\det(L')|$ for a certain primitive sublattice $L'$ of $L$ is satisfied if the determinant $|\det(L)|$ is big enough.
\end{oss}

We now furnish an effective criterion to prove that a certain K3 surface $X$ has infinitely many elliptic fibrations with an infinite number of sections.

\begin{lemma} \label{lemma:aut}
Let $X$ be an elliptic K3 surface, $|F|$ an elliptic fibration on $X$, and $f\in \Or^+(\NS(X))$.
\begin{enumerate}
    \item If $E=f(F)$ is nef, then there exists $s\in W<\Or^+(\NS(X))$ such that $g=f\circ s$ preserves the nef cone and $g(F)=E$.
    \item If $|F|$ has only irreducible fibers and $E=f(F)$ is nef, then $f$ preserves the nef cone.
    \item If $f$ preserves the nef cone and the set of Hodge isometries $\Or_{Hdg}(\T(X))=\{\pm \id\}$ is trivial, then $f$ corresponds to an automorphism of $X$ if and only if $\pm \id=\overline{f}\in \Or(A_{\NS(X)})$.
\end{enumerate}
\end {lemma}
\begin{proof}
\begin{enumerate}
    \item Let $C$ be an effective $(-2)$-curve; $f^{-1}$ is an isometry, so $(f^{-1}(C))^2=-2$, hence by Riemann-Roch either $f^{-1}(C)$ or $-f^{-1}(C)$ is effective. Assume that $D=-f^{-1}(C)$ is effective. Then
    $$0 \ge -DF= f^{-1}(C)F=Cf(F)\ge 0,$$
    so necessarily $DF=0$, or equivalently $D$ is contained in a reducible fiber of the elliptic pencil $|F|$. Composing $f$ with the reflection $s_D$ across $D$ yields $f'=f\circ s_D$ such that $f'(F)=f(F)=E$ and $(f')^{-1}(C)=s_D(-D)=D$. Repeating the process we obtain a $g=f\circ s\in \Or^+(\NS(X))$ such that $g(F)=E$ and $g^{-1}$ preserves the set of effective $(-2)$-curves (the process ends since there are only finitely many reducible fibers in the elliptic pencil $|F|$). Now $g$ preserves the nef cone, since if $D$ is a nef element,
    $$f(D)C=Df^{-1}(C)\ge 0$$
    for all effective $(-2)$-curves, since $f^{-1}(C)$ is effective.
    \item It is a special case of the previous point, since there are no $(-2)$-curves orthogonal to $F$.
    \item By Proposition \ref{prop:tutto} a power of $f$ corresponds to an automorphism of $X$. Then $f$ preserves the set of smooth $(-2)$-curves, since one of its powers does and $f$ preserves the nef cone. Therefore, again by Proposition \ref{prop:tutto}, $f$ corresponds to an automorphism of $X$ if and only if $\overline{f}\in \Or(A_{\NS(X)})$ coincides with the restriction of a $g \in \Or_{Hdg}(\T(X))=\{\pm \id\}$.
\end{enumerate}
\end{proof}

\begin{oss} \label{oss:general}
From \cite{oguiso2}, Lemma 4.1, we know that the assumption $\Or_{Hdg}(\T(X))=\{\pm \id\}$ is always satisfied if $X$ has odd Picard rank. Moreover, if $X$ has even Picard rank and the period $\omega_X\in \T(X)_\C$ is very general, then again $\Or_{Hdg}(\T(X))=\{\pm \id\}$. Indeed, any Hodge isometry of the trascendental lattice $\T(X)$ has $\omega_X$ as an eigenvector, so it suffices to choose $\omega_X$ outside the countable union of lines in $\T(X)_\C$ corresponding to the eigenvectors of isometries in $\Or(\T(X))$.
\end{oss}

\begin{teo} \label{teo:unique}
Let $X$ be a K3 surface, $|F|$ an elliptic fibration on $X$ inducing the decomposition $\NS(X)=U\oplus L$. If the Picard rank $\rho(X)$ is even, assume that the period $\omega_X\in \T(X)_\C$ is very general. Then $X$ admits a unique elliptic fibration up to automorphisms if and only if $L$ is unique in its genus and the restriction map $\Or(L)\rightarrow \Or(A_L)$ is surjective.
\end{teo}
\begin{proof}
First assume that the restriction map $\Or(L)\rightarrow \Or(A_L)$ is not surjective, and let $\varphi\in \Or(A_L)$ be an isometry of $A_L$ not in the image of the restriction map. By \cite{huybrechts1}, Theorem 14.2.4 we have an $f\in \Or(\NS(X))$ such that $\overline{f}=\varphi \in \Or(A_L)=\Or(A_{\NS(X)})$. Up to composing $f$ with a finite number of elements in $W$, we can assume that $E=f(F)$ is nef, and hence that it induces an elliptic fibration. Notice that the Weyl group $W$ acts trivially on $A_{\NS(X)}$, so we still have that $\overline{f}=\varphi \in \Or(A_L)=\Or(A_{\NS(X)})$. We want to prove that $E$ and $F$ induce distinct elliptic fibrations under the action of $\Aut(X)$. Assume by contradiction that there exists $g\in \Aut(X)$ such that $g^*(F)=E$. Then $h=(g^*)^{-1}\circ f$ preserves the elliptic fibration $|F|$. Up to composing with a translation in $\MW(F)$, we can assume that $h$ preserves the lattice $U$ generated by $F$ and its zero section; hence $h\in \Or(L)$ is an isometry of the orthogonal complement of $U$. By the generality assumption on $\omega_X$, Remark \ref{oss:general} and Lemma \ref{lemma:aut} imply that $\overline{(g^*)^{-1}}=\pm \id \in \Or(A_{\NS(X)})=\Or(A_L)$. Hence $\overline{h}=\pm \overline{f}=\pm \varphi$ does not lift to an isometry of $L$, a contradiction.\\
Secondly assume that $L$ is not unique in its genus, and let $S$ be a lattice in the genus of $L$ not isomorphic to $L$. By \cite{nikulin9}, Proposition 1.5.1, we have an embedding $j:S \hookrightarrow \NS(X)$ such that $j(S)^\bot = U$. Assume that $j(S)^\bot= \langle E,C \rangle$, where $E^2=0$, $C^2=-2$ and $EC=1$. After applying a certain (finite) number of isometries $s_i\in W < \Or(\NS(X))$, we can assume that $E$ is nef, and induces an elliptic fibration on $X$ with respect to which $\NS(X)=\langle E,C \rangle \oplus S$, since the elements in the Weyl subgroup do not change the intersections. $C$ is effective, since $EC=1>0$, so $E$ has at least a section $S_E$ (a certain irreducible component of $C$). We want to prove that $S$ is isometric to the orthogonal complement of $\langle E,S_E\rangle$. Since $C$ is an effective $(-2)$-curve, we can write
$$C=S_E + \sum_{i,j}{C_i^{(j)}},$$
with $C_i^{(j)}$ a vertical $(-2)$-curve for every $i,j$, and $j$ indexing the reducible fibers of the fibration induced by $E$. Since
$$-2=C^2=S_E^2+2S_E \left(\sum_{i,j}{C_i^{(j)}}\right) + \sum_{j}{\left(\sum_i{C_i^{(j)}}\right)^2},$$
and the intersection form restricted to the $(-2)$-curves of a reducible fiber not intersecting $S_E$ is negative definite, we conclude that for all $j$ there exists one and only one $i$ such that $S_E  C_i^{(j)}=1$. Hence applying the reflections $s_{C_i^{(j)}} \in W$ we keep $E$ fixed and we map $S_E$ into $S_E'=S_E+\sum_j{C_i^{(j)}} $. Using the same argument for $S_E'$, we conclude that $C$ and $S_E$ are conjugated under the action of the Weyl group $W$, and so $\langle E,S_E\rangle^\bot\cong \langle E,C\rangle^\bot \cong S$. Certainly $E,F$ are distinct up to automorphism, since the two orthogonal complements $L,S$ are not isometric.\\
Finally we have to prove the converse. So assume that $L$ is unique in its genus and the restriction map $\Or(L)\rightarrow \Or(A_L)$ is surjective. Let $|E|$ be another elliptic fibration on $X$. By the same reasoning as above, $U_E^\bot=\langle E,S_E\rangle^\bot \subseteq \NS(X)$ is in the genus of $L$, hence it is isometric to $L$ by assumption. This gives us an $f\in \Or^+(\NS(X))$ such that $f(F)=E$. The restricted isometry $\overline{f}\in \Or(A_L)$ comes by assumption from a $\varphi\in \Or(L)$, so we obtain $g\in \Or^+(\NS(X))$ such that $g|_{U_E}=\id$ and $g|_L=\varphi$. Now $h=f \circ g^{-1} \in \Or^+(\NS(X))$ is such that $h(F)=E$ and $\overline{h}=\id \in \Or(A_L)$, so by Lemma \ref{lemma:aut} we have a $h'=h\circ s\in \Or^+(\NS(X))$ such that $h'(F)=E$, $h'$ preserves the nef cone and $\overline{h'}=\id \in \Or(A_L)$, so $h'$ is an automorphism of $X$.
\end{proof}

\begin{oss}
Notice that the converse implication does not need the generality assumption on the period.
\end{oss}

In a completely analogous manner we can prove:

\begin{teo} \label{teo:unique2}
Let $X$ be a K3 surface with $\NS(X)=U\oplus L$, where $L$ is not a root-overlattice. If the Picard rank $\rho(X)$ is even, assume that the period $\omega_X\in \T(X)_\C$ is very general. Then $X$ admits a unique elliptic fibration with infinitely many sections up to automorphisms if and only if $L$ is the unique non-root-overlattice in its genus and the restriction map $\Or(L)\rightarrow \Or(A_L)$ is surjective.
\end{teo}

\begin{oss}
As a corollary we have that, if $X$ is a K3 surface with an elliptic fibration $|F|$ with infinitely many sections, and $\NS(X)=U \oplus L$ is such that the restriction map $\Or(L)\rightarrow \Or(A_L)$ is not surjective, then $X$ has positive entropy (even without the generality assumption). Indeed, having positive entropy only depends on the N\'eron-Severi lattice $\NS(X)$, so this just follows from the previous result.
\end{oss}

We denote
$$N(X)=\#\{\text{elliptic fibrations on $X$}\}/\Aut(X),$$
$$N^{\text{pos}}(X)=\#\{\text{elliptic fibrations on $X$ with infinitely many sections}\}/\Aut(X).$$
Both these numbers are finite, since $N^{\text{pos}}(X) \le N(X)$ and $N(X)$ is always finite by \cite{sterk}, Proposition 2.6. Then Theorem \ref{teo:unique} identifies K3 surfaces with $N(X)=1$, while Theorem \ref{teo:unique2} identifies K3 surfaces with $N^{\text{pos}}(X)=1$. The condition $N^{\text{pos}}(X)>1$ is clearly sufficient for $X$ to have positive entropy. Moreover:

\begin{prop} \label{prop:rho10}
Let $X$ be a K3 surface with $\NS(X)=U\oplus L$, $L$ without roots, and $\rho(X) \le 10$. Then $N^{\text{pos}}(X)=N(X)$. In particular, if $L$ is not unique in its genus, $X$ has positive entropy.
\end{prop}
\begin{proof}
This follows immediately from Proposition \ref{prop:rootgenus}, since elliptic fibrations with a finite number of sections are induced by those $E\in \NS(X)$ such that $\langle E,S_E \rangle^\bot$ is a root-overlattice.
\end{proof}

This becomes particularly powerful in view of the following result. A lattice is said \textit{primitive} if the greatest common divisor of the entries of its intersection matrix (with respect to any basis) is $1$.

\begin{teo}[\cite{watson2}, Theorem 1 - \cite{nebe}] \label{teo:watson}
Let $L$ be a definite lattice of rank $r\ge 2$ (not necessarily even). Assume that $L$ is unique in its genus. Then $r\le 10$, and there exists a complete (and finite) list of all the definite, primitive lattices of rank $2\le r \le 10$ unique in their genus.
\end{teo}

This naturally divides our work into three parts:

\begin{enumerate}
    \item Classify K3 surfaces of zero entropy and Picard rank $3$.
    \item If $4\le \rho(X) \le 10$, the set of N\'eron-Severi lattices of K3 surfaces of zero entropy is a subset of the lattices of the form $U \oplus L$, where $L$ varies among the multiplies of the lattices in Watson's list. The list is available online at \cite{nebe}.
    \item If $\rho(X)>10$ then, except for very few cases in rank $11$ and $12$, every $L$ has many non-isometric lattices in its genus. However, some of them could be root-overlattices.
\end{enumerate}

In the next three sections we will analyze these three cases individually, in order to obtain a list of all the K3 surfaces of zero entropy admitting an elliptic fibration with only irreducible fibers.

\section{K3 surfaces of Picard rank $3$}

As we have already observed, a K3 surface admitting an elliptic fibration with infinitely many sections must have Picard rank at least $3$. \newline

In this section let $X$ be an elliptic K3 surface of Picard rank $\rho(X)=3$, and denote by $|F|$ an elliptic fibration on $X$ of maximal rank. The Néron-Severi lattice has the form
$$\NS(X)=U \oplus \langle -2k \rangle$$
for a certain $k \ge 2$, since the fibration $|F|$ has no reducible fibers. The goal of the section is to find the values of $k$ for which $X$ has zero entropy. Shimada in \cite{shimada1} presents an algorithm to compute the automorphism group of these K3 surfaces; our approach achieves less for a fixed $k$, since in most cases we are not able to describe the whole $\Aut(X)$ completely, but it gives informations about these surfaces for all $k$ at once.\newline

Most results proved in this section are contained in Nikulin's paper \cite{nikulin8}; however, since some of the ideas used will be useful later on, we have decided to include the proofs. Moreover, our approach is rather different from Nikulin's. Of course the classification of K3 surfaces of Picard rank $3$ admitting a unique elliptic (resp. genus $1$) fibration we independently obtain coincides with Nikulin's (cf. Theorem 3 and the subsequent discussion in \cite{nikulin8}).

\begin{lemma} \label{lemma:facile}
There exist isomorphisms
$$\Aut(X,|F|) \cong \MW(F) \rtimes \Z/2\Z \cong \Z/2\Z * \Z/2\Z.$$
In particular $\Aut(X)$ is infinite.
\end{lemma}
\begin{proof}
We can check with a straightforward computation that $\Aut(X,|F|)$ is generated by the isometries
$$\tau=\begin{pmatrix}
1 &k &2k\\
0 &1 &0\\
0 &-1 &-1
\end{pmatrix}, \qquad \sigma=\begin{pmatrix}
1 &0 &0\\
0 &1 &0\\
0 &0 &-1
\end{pmatrix},$$
and that $\MW(F)=\langle \tau \rangle$.
\end{proof}

For the sake of readability, we rewrite Lemma \ref{lemma:int} and equation (\ref{eq:div}) in this setting.

\begin{lemma} \label{lemma:int3}
Let $F\ne E=[\alpha,\beta,\gamma],C=[x,y,z]\in \NS(X)$ be effective, primitive divisors such that $E^2=0$ and $C^2=-2$. Then we have
\begin{equation} \label{eq:div3}
\beta \mid k\gamma^2, \qquad y \mid kz^2-1.
\end{equation}
Moreover the equation $EC=m$ can be equivalently written
\begin{equation} \label{eq:m}
k(y\gamma - \beta z)^2=\beta(\beta+my).
\end{equation}
In particular $E$ is nef if and only if 
$$k(y\gamma - \beta z)^2 - \beta^2 \ge 0$$
for any such $C$, and $E$ induces a genus $1$ fibration with only irreducible fibers if and only if
$$k(y\gamma - \beta z)^2 - \beta^2 > 0$$
for any such $C$.
\end{lemma}

\begin{oss} \label{oss:ext}
Let $|E|\ne |F|$ be another elliptic fibration on $X$. Then the trivial lattice $\langle E,S_E\rangle$ gives a primitive embedding $i:U \hookrightarrow \NS(X)$. Since $\langle -2k\rangle$ is unique in its genus we have that $i(U)^\bot \cong \langle -2k \rangle$, thus $E$ induces an isometry $f\in \Or^+(\NS(X))$ such that $f(F)=E$. In particular every elliptic fibration on $X$ is of maximal rank.
\end{oss}

We can now apply Theorem \ref{teo:unique}:

\begin{prop}
Let $X$ be an elliptic K3 surface with $\NS(X)=U \oplus \langle -2k \rangle$, for $k \ge 2$. Denote by $m$ the number of distinct prime divisors of $k$. Then the number of elliptic fibrations on $X$ (of maximal rank) up to automorphisms is $2^{m-1}$. In particular $X$ has a unique elliptic fibration up to automorphisms if and only if $k$ is a power of a prime.
\end{prop}
\begin{proof}
By Remark \ref{oss:ext}, there exists a function
$$\{|E| \text{ elliptic fibration}\} \longrightarrow \{ f\in \Or^+(\NS(X))\}.$$
Composing with the restriction map $\Or(\NS(X))\rightarrow \Or(A_{\NS(X)})=\Or(A_L)$, we obtain another function
$$\{|E| \text{ elliptic fibration}\} \longrightarrow \{ \overline{f}\in \Or(A_L)\}.$$
By the proof of Theorem \ref{teo:unique} this map is surjective. Two elliptic fibrations $|E_1|,|E_2|$ are conjugated under the action of $\Aut(X)$ if and only if there exists $g \in \Aut(X)$ such that $g(E_1)=E_2$ thus, by Lemma \ref{lemma:aut}, if and only if the induced $\overline{f_1},\overline{f_2}\in \Or(A_L)$ satisfy $\overline{f_1}=\pm \overline{f_2}$. Consequently we obtain a bijection
$$\{\text{elliptic fibrations}\}/\Aut(X) \stackrel{\sim}{\longrightarrow} \Or(A_L)/\{\pm \id\}.$$
The discriminant group $A_L$ is cyclic, generated by the element $\frac{D}{2k}$, where $\{D\}$ is a basis for $L=\langle -2k\rangle$. Its norm in $A_L$ is $-\frac{1}{2k} \pmod{2\Z}$, hence we can identify $\Or(A_L)$ with the group
$$G_k=\{x \in \Z/2k\Z \mid x^2\equiv 1 \pmod{4k}\}.$$
An immediate application of the Chinese remainder theorem shows that $G_k$ has $2^m$ elements, concluding the proof.
\end{proof}

\begin{cor} \label{cor:primo3}
Let $X$ be a K3 surface with $\NS(X)=U\oplus \langle -2k \rangle$. If $k$ is not a power of a prime, then $X$ has positive entropy.
\end{cor}

This is the most we can obtain using the general theory of lattices. Therefore, in order to conclude the classification of K3 surfaces of zero entropy in Picard rank $3$, we have to work explicitly on the N\'eron-Severi lattice.\newline 

The case when $k=p^n$ is a power of a prime is clearly far more involuted, since $|F|$ is always the unique elliptic fibration up to automorphisms. We begin with a preliminary result concerning the possible elliptic fibrations on $X$.

\begin{lemma} \label{lemma:choice}
Assume that $\NS(X)=U\oplus \langle -2k \rangle$, with $k=p^n$. Let $E\in \NS(X)$ be effective and primitive with $E^2=0$ and $C\in \NS(X)$ with $C^2=-2$ and $EC=1$. Then $E$ can only be of two types:
$$F'_{q,\gamma'}=[q^2+k\gamma'^2,q^2,q\gamma'], \qquad F''_{q,\gamma'}=[q^2k+\gamma'^2,q^2k,q\gamma'],$$
with $q>0$ and $(q,\gamma')=1$.
\end{lemma}
\begin{proof}
Let $E=[\alpha,\beta,\gamma]$, and put $q=(\beta,\gamma)>0$. Since $q \nmid \alpha$ and $\alpha=\beta+\frac{k\gamma^2}{\beta}$, we have that $q^2 \mid \beta$; say $\beta=q^2\beta'$, $\gamma=q\gamma'$, with $(\beta',\gamma')=1$. Let us distinguish two cases.
\begin{itemize}
    \item If $p\mid q$, then $p$ divides $\beta,\gamma$, hence $p$ does not divide
    $$\alpha=q^2\beta'+\frac{p^n\gamma'^2}{\beta'}.$$
    Since $\beta' \mid p^n$, we can only have $\beta'=p^n=k$.
    \item If $(p,q)=1$, then as above $\beta' \mid p^n$, say $\beta'=p^m$. But if $0 < m < n$, then equation (\ref{eq:m})
    $$p^n(\gamma' y - q \beta' z)^2=\beta' (q^2 \beta' + y)$$
    implies that $p \mid q^2\beta'+y$, thus $p \mid y$, and this is a contradiction since $p \mid y \mid p^n z^2-1$.
    \end{itemize}
Notice that the condition $(q,\gamma')=1$ is necessary for $E$ to be primitive.
\end{proof}

The first result towards the classification of elliptic K3 surfaces of Picard rank $3$ and zero entropy is the following:

\begin{prop} \label{prop:parteuno}
Assume that $\NS(X)=U\oplus \langle -2k \rangle$, with $k\ge 2$. Suppose that it exists $q\ge 2$ such that $q^2 < k$ and $q \nmid k-1$. Then $X$ has infinitely many elliptic fibrations, or equivalently it has positive entropy.
\end{prop}
\begin{proof}
If $k$ is not a power of a prime, then $X$ has positive entropy by Corollary \ref{cor:primo3}. Therefore assume that $k=p^n$ is the power of a prime. We distinguish two cases depending on the exponent $n$.
\begin{itemize}
    \item[$n\ge 3$:] Consider the divisor $E=[p+p^{n-1},p,1]$ on $X$. We claim that $E$ induces a genus $1$ fibration on $X$ with infinite automorphism group. First, $E$ is nef: indeed by Lemma \ref{lemma:int3} it amounts to showing that
    $$p^n (y - pz)^2\ge p^2$$
    for any effective $(-2)$-curve $C=[x,y,z]$. Since $y \mid p^n z^2-1$ by equation (\ref{eq:div3}), $y$ must be coprime with $p$. Therefore $y - pz \ne 0$ and
    $$p^n (y - pz)^2 \ge p^n \ge p^3 \ge p^2,$$
    as claimed. Notice that $E$ does not admit a section, as the intersection $EC$ is a multiple of $p$ for any curve $C$ in $X$. However the degree of the genus $1$ fibration $|E|$ (cf. \cite{huybrechts1}, Definition 11.4.3) is $p$, since $E$ admits the $p$-section $S_p=[p^{2n-4}+2p^{n-2},p^{n-2}-1,p^{n-3}]$. Therefore we can consider the corresponding Jacobian fibration $J(X)$, which has N\'eron-Severi lattice isometric to $U\oplus \langle -2p^{n-2}\rangle$ by \cite{keum}, Lemma 2.1. Since $n\ge 3$ by assumption, $E$ becomes an elliptic fibration $J(E)$ with only irreducible fibers on $J(X)$ by Remark \ref{oss:ext}. Consequently the group $\Aut(J(E))$ is infinite, and it acts with infinite order on $E$: indeed, under the identifications $J(X)\cong \mathrm{Pic}^0(X/\Prj^1)$ and $X\cong \mathrm{Pic}^1(X/\Prj^1)$ (see \cite[Section 11.4.1]{huybrechts1}), the generic fiber of $J(X)$ acts by translation on $X$ and it preserves the fibration $|E|$. This shows that $\Aut(E)$ is infinite. We conclude that $X$ has positive entropy by Theorem \ref{thm:og}.
    \item[$n \le 2$:] By assumption there exists a $q$ such that $q^2 < k=p^n$ and $q \nmid p^n-1$. Since $n\le 2$, $q$ must be coprime with $p$. Consider the primitive isotropic divisor $E=[q^2+p^n,q^2,q]$. We claim that $E$ induces an elliptic fibration on $X$. In order to show that it is nef, by Lemma \ref{lemma:int3} it amounts to showing that
    $$p^n(yq-q^2z)^2\ge q^4$$
    for any effective $(-2)$-curve $C=[x,y,z]$. Since $p^n > q^2$, it suffices to show that $y-qz\ne 0$. Therefore assume by contradiction that $y=qz$. By equation (\ref{eq:div3}) we have that $y=qz \mid p^n z-1$, so necessarily $z=1$. But then $y=q \mid p^n -1$, contradicting the assumption on $q$. Then we have to show that $|E|$ has a section, or equivalently that its degree is $1$. Since $\det(\NS(X))=2k=2p^n$ with $n\le 2$, by \cite{keum}, Lemma 2.1 it is sufficient to show that the degree of $|E|$ is not $p$. In fact $E C$ is never a multiple of $p$ if $C$ has square $-2$: indeed, if $ C =m$, by Lemma \ref{lemma:int3} we have that
    $$p^n(y-qz)^2=q^2+my,$$
    so if $p\mid m$ we would have that $p\mid q$, a contradiction.
\end{itemize}
\end{proof}

The previous proposition deals with the $k$ that do not satisfy the condition
\begin{itemize}
    \item[(C):] For all $r\in \N$ with $r^2<k$, $r \mid k-1$.
\end{itemize}
We can list all the natural numbers satisfying (C):

\begin{lemma}
The only natural numbers $k$ satisfying (C) are
$$\Lc_1=\{2,3,4,5,7,9,13,25\}.$$
\end{lemma}
\begin{proof}
Let $N$ be the number of distinct prime divisors of $k-1$, and assume $N\ge 5$. Put 
$$e=\frac{1}{2}\log_2{k}.$$
We have that $N\le e$, since otherwise, denoted by $\{p_i\}$ the increasing sequence of prime numbers, we would have
$$\prod_{i=1}^N{p_i}>4^N=2^{2N}>2^{2e}=k>k-1,$$
(remember that we have $N\ge 5$, and $2\cdot 3\cdot 5 \cdot 7 \cdot 11 > 4^5$), contradicting the fact that $k-1$ has $N$ distinct prime divisors. Now let $q$ be the smallest prime number not dividing $k-1$. If we can show that $q^2<k$, we are done. $k-1$ has $N$ distict prime divisors, so $q$ is smaller or equal than the $(N+1)$-th prime number, which in turn is strictly smaller than $2^N$ (since there is always a prime number between $\alpha$ and $2\alpha$ for every $\alpha>1$). Hence
$q^2< (2^N)^2\le 2^{2e}=k. $
If instead $N\le 4$, then as above we can choose $q$ as one of the first $5$ prime numbers, hence $q\le 11$. Therefore all natural numbers strictly greater than $11^2=121$ cannot satisfy (C). A quick inspection of the first $121$ natural numbers yields the list $\Lc_1$ above.
\end{proof}

Finally it only remains to deal with a finite number of cases; the next proposition is the converse of Proposition \ref{prop:parteuno}:

\begin{prop} \label{prop:conti3}
Assume $\NS(X)=U\oplus \langle -2k \rangle$, and suppose that $k$ satisfies the condition (C) (or equivalently, $k\in \Lc_1$). Then none of the $F'_{q,\gamma'}$ and $F''_{q,\gamma'}$ in Lemma \ref{lemma:choice} is nef, hence $X$ has zero entropy.
\end{prop}
\begin{proof}
The two proofs for $F'_{q,\gamma'}$ and $F''_{q,\gamma'}$ are similar; we will write the first one with all the details, and the reader can easily complete the latter following the steps of the former.
\begin{enumerate}
\item Let $E=F'_{q,\gamma'}$. To deny the nefness of $E$, by Lemma \ref{lemma:int3} we want to show that there exists an $r>0$ and an effective $C=[x,y,z]$ with $C^2=-2$ such that
$$k(\gamma' y - q z)^2=q^2-ry.$$
Put $-m=\gamma' y -qz$, thus
$$y=\frac{q^2-km^2}{r}=\frac{qz-m}{\gamma'}$$
and
$$\gamma'(q^2-km^2)=r(qz-m).$$
Looking at the equality modulo $q$, we choose $m$ such that $km\gamma' \equiv r \pmod {q}$; more precisely put
\begin{equation} \label{eq:sys}
\begin{cases}
km\gamma'=r+\eta q\\
m=\alpha r +\delta q
\end{cases}
\end{equation}
where $\alpha \in (-\frac{q}{2},\frac{q}{2})$ is the inverse of $k\gamma'$ modulo $q$. Notice that the choice of $r,\delta$ depends on $\alpha$, in order to ensure that $ry=q^2-km^2>0$: however, we can always choose $r$ with $r^2<k$, hence $k\equiv 1 \pmod{r}$ by assumption on $k$. This is easy to see for any case, and we only show how this can be achieved in the case $k=25$:
\begin{table}[H]
    \centering
    \begin{tabular}{c|c|c}
         $\alpha$ &$r$ &$\delta$\\
         \hline
         $\pm (0, \frac{q}{5})$ &$1$ &$0$\\
         $\pm (\frac{2}{5}q, \frac{q}{2})$ &$2$ &$\mp 1$\\
         $\pm (\frac{4}{15}q, \frac{2}{5}q)$ &$3$ &$\mp 1$\\
         $\pm (\frac{q}{5},\frac{4}{15}q)$ &$4$ &$\mp 1$
    \end{tabular}
\end{table}
Notice that, by Equation (\ref{eq:sys}), either $r=1$ or $m \equiv \pm q \pmod{r}$. 
It remains to prove that, with these choices, $y\in \N$, $z\in \Z$, and $y \mid kz^2-1$. By construction $y >0$, and
$$q^2-km^2 \equiv q^2-m^2 \equiv 0 \pmod{r},$$
hence $y \in \N$. Now
$$z=\frac{\gamma'y+m}{q}=\frac{\gamma'(q^2-km^2)+rm}{rq}=\frac{\gamma'q^2-rm-m\eta q+rm}{rq}=\frac{\gamma' q -\eta m}{r}$$
and
\begin{equation*}
\begin{split}
kz^2-1 &=\frac{k\gamma'^2q^2-2k\gamma' q \eta m +k\eta^2 m^2-r^2}{r^2}=\frac{\eta^2(km^2-q^2)+q^2\eta^2-2q\eta(r+\eta q)+k\gamma'^2q^2-r^2}{r^2}=\\
&=\frac{\eta^2(km^2-q^2)-q^2\eta^2-2q\eta r+k\gamma'^2q^2-r^2}{r^2}=\frac{\eta^2(km^2-q^2)-(q \eta +r)^2+k\gamma'^2 q^2}{r^2}=\\
&=\frac{\eta^2(km^2-q^2)+k \gamma'^2(-km^2+q^2)}{r^2})= y \frac{k\gamma'^2-\eta^2}{r}.
\end{split}
\end{equation*}
thus the claim holds if $r=1$. If instead $r>1$, we have from (\ref{eq:sys}) that
$$k\gamma'\alpha r \pm k\gamma' q=r +\eta q,$$
from which
\begin{equation} \label{eq:brutto}
    (\eta \mp \gamma')q \equiv 0 \pmod{r},
\end{equation}
and this assures that $z\in \Z$, since $m \equiv \pm q \pmod{r}$. Moreover, if $(q,r)=1$, (\ref{eq:brutto}) also gives that
$$k\gamma'^2-\eta^2\equiv \gamma'^2-\eta^2\equiv 0 \pmod{r};$$
if instead $(q,r)>1$, let us consider $r=2$, as the other cases are analogous. Then $q$ is even, and $\gamma'$ is odd. Now
$$q \eta = km\gamma'-r=(k\gamma' \alpha -1)r \pm k\gamma'q,$$
and if $2^l$ divides exactly $q$, it divides exactly $k\gamma' q$, but $q \mid k\gamma' \alpha -1$, thus $2^{l+1} \mid (k\gamma' \alpha -1)r$, from which $2^l$ divides exactly $q \eta$, i.e $\eta$ is odd. Therefore $y \mid kz^2 -1$, as claimed.
\item Let $E=F''_{q,\gamma'}$. Now the equation is
$$(\gamma' y - qk z)^2=q^2k-ry.$$
If $-m=\gamma'y-qkz$, then
$$y=\frac{q^2k-m^2}{r}=\frac{qkz-m}{\gamma'}$$
and
\begin{equation} \label{eq:brutto2}
    \begin{cases}
\gamma' m=r+\eta q\\
m=\alpha r + \delta q
\end{cases}
\end{equation}
with $\alpha\in (-\frac{q}{2},\frac{q}{2})$ the inverse of $\gamma'$ modulo $q$.\\
We need to ensure again that $y\in \N$, $z\in \Z$ and $y \mid kz^2-1$. We let again $r$ vary in $(0,\sqrt{k})$ and $\delta \in [-\sqrt{k},\sqrt{k}]$, so that $k \equiv 1 \pmod{r}$. Now
$$z=\frac{\gamma'y+m}{qk}=\frac{\gamma'(q^2k-m^2)+mr}{rqk}=\frac{\gamma' q^2k - rm - \eta q m +rm}{rqk}=\frac{\gamma'q k - \eta m}{rk}.$$
Notice that, since $(r,k)=1$, $k=p^n$ is coprime with at least one of $\eta,m$ (otherwise $p \mid r$ by Equation (\ref{eq:brutto2})); however we are going to prove later that we can let $k$ divide either $\eta$ or $m$ (say respectively $(\eta,m)=(k\eta',m')$ or $(\eta,m)=(\eta',km')$), so $z \in \Z$ if and only if 
$$\gamma' q - \eta' m' \equiv 0 \pmod{r}.$$
Notice however that $\eta'\equiv \eta \pmod{r}$ and $m' \equiv m \pmod{r}$. Reasoning as in the first case, we obtain easily that $y,z \in \Z$ if $\delta \equiv \pm 1 \pmod{r}$. Again analogously as above we see that $y \mid kz^2-1$ if and only if
$$\gamma'^2-k\eta'^2\equiv 0 \pmod{r},$$
and the argument above adapts perfectly. It only remains to prove that we can always arrange $r\in (0,\sqrt{k})$ and $\delta\in [-\sqrt{k},\sqrt{k}]$ such that
\begin{equation}
    \begin{cases} \label{eq:sys2}
\delta \equiv 1 \pmod{r}\\
q^2k-m^2>0\\
k \mid \eta \text{ or } k \mid m
\end{cases}
\end{equation}
Recall that by equation (\ref{eq:brutto2})
$$\eta m q =m(\gamma'm-r)=(\alpha r +\delta q)((\gamma'\alpha-1)r+\gamma'\delta q).$$
Assume $(k,q)=1$. Then we want $k$ to divide $m$, that is
$$\delta \equiv \underbrace{(-\alpha q^{-1})}_{=\omega}r \pmod{k},$$
and we want the other conditions in (\ref{eq:sys2}) to be satisfied.
It is again not hard to see that this is actually possible for every value of $\omega$ and for any $k$, and we show it explicitly for $k=25$:
\begin{center}
\begin{tabular}{c|c|c}
    $\omega$ &$r$ &$\delta$\\
    \hline
    0 &1 &1\\
    1 &1 &1\\
    2 &1 &2\\
    3 &1 &3\\
    4 &1 &4\\
    5 &$1/4$ &$\pm 5$\\
    6 &4 &$-1$\\
\end{tabular}\qquad
\begin{tabular}{c|c|c}
    $\omega$ &$r$ &$\delta$\\
    \hline
    7 &4 &3\\
    8 &3 &$-1$\\
    9 &3 &2\\
    10 &$2/3$ &$\pm 5$\\
    11 &2 &$-3$\\
    12 &2 &$-1$
\end{tabular}
\end{center}
Obviously for $-\omega$ we choose the same $r$ as for $\omega$ and $\delta$ opposite in sign. For $\omega=5,10$ there are $2$ cases, depending on whether $\alpha$ is positive or negative: for instance, if $\omega=5$ and $\alpha>0$, we choose $(r,\delta)=(4,-5)$, while if $\alpha <0$ we choose $(r,\delta)=(1,5)$.

If instead $(q,k)>1$, we proceed analogously to show that we can make $k$ divide $\eta$: first of all we divide the expression $\eta q=(\gamma'\alpha-1)r+\gamma'\delta q$ by $q$, obtaining
$$\eta = \frac{\gamma'\alpha-1}{q}r+\gamma' \delta .$$
Then we apply exactly the same argument as before with $\omega\equiv -(\frac{\gamma'\alpha-1}{q})\gamma'^{-1} \pmod{k}$, and we are done.
\end{enumerate}
\end{proof}

We have proven:

\begin{teo} \label{teo:finale3}
Let $X$ be an elliptic K3 surface of Picard rank $3$, with $\NS(X)=U\oplus \langle -2k\rangle$, $k \ge 2$. Then the following are equivalent:
\begin{enumerate}
    \item $X$ has zero entropy;
    \item $X$ admits a unique elliptic fibration $|F|$ (of maximal rank);
    \item $\Aut(X)=\Aut(X,|F|)=\langle \tau,\sigma \rangle$, where $\tau,\sigma$ are defined in Lemma \ref{lemma:facile};
    \item For all $r\in \N$ with $r^2<k$, $r$ divides $k-1$;
    \item $k\in \Lc_1= \{2,3,4,5,7,9,13,25\}$.
\end{enumerate}
\end{teo}

A natural question is whether the elliptic K3 surfaces with $\NS(X)$ of one of these $8$ types admit other genus $1$ fibrations. As a corollary of the previous theorem, any other genus $1$ fibration must have no sections. Recall that, if $|E|$ is a genus $1$ fibration on $X$, we can define the \textit{degree} of $|E|$ as the minimum positive intersection $EC$, with $C$ varying among curves on $X$ (cf. \cite{keum} or \cite{huybrechts1}, Definition 11.4.3). $|E|$ is an elliptic fibration if and only if it has degree $1$. Moreover, if the degree $d$ of a genus $1$ fibration is greater than $1$, its associated Jacobian fibration $J(X)$ satisfies the property $\det(\NS(X))=d^2 \det(\NS(J(X))$ (cf. \cite{keum}, Lemma 2.1).

\begin{prop}
Let $X$ be an elliptic K3 surface of Picard rank $3$, with $\NS(X)=U\oplus \langle -2k\rangle$ and $k \in \{2,3,4,5,7,9,13,25\}$. Denote by $F$ the fiber of the given elliptic fibration. If $k\in \{2,3,5,7,13\}$ is prime, then $X$ admits a unique genus $1$ fibration, induced by $F$.

If instead $k=p^2\in \{4,9,25\}$ is a square, then a primitive element $F\ne E\in \NS(X)$ with $E^2=0$ induces a genus $1$ fibration if and only if $EF=p$.
\end{prop}
\begin{proof}
As observed above, any genus $1$ fibration $|E|$ on $X$ has no sections, so its degree is greater than $1$. From the discussion above, there cannot be any genus $1$ fibrations on $X$ if $\det(\NS(X))=2k$ is square-free. But if $k=2$ and there exists a genus $1$ fibration on $X$ with no sections, then necessarily it must have degree $2$ and $\det(\NS(J(X))=1$, which is impossible (there are no even unimodular lattices of rank $3$).

Assume instead that $k=p^2\in \{4,9,25\}$ is a square of a prime. If $E=[\alpha,\beta,\gamma]$ induces a genus $1$ fibration on $X$, reasoning as in Lemma \ref{lemma:choice} we have that $\beta=q^2\beta'$, with $\beta'=p$ (indeed, if $\beta'=1$ or $k=p^2$, then $E$ would be one of the $F',F''$, and we have proven in Proposition \ref{prop:conti3} that these elements are not nef). Assume $k=4$, so $\beta=2q^2$ and $\gamma=q\gamma'$, with $(q,\gamma')=1$. Then, if $E$ is nef, by Lemma \ref{lemma:int3} we have
$$4(2qz-\gamma')^2-4q^2 \ge 0$$
for all $z\in \Z$ (this is just the intersection of $E$ with the sections of $F$, divided by $q^2$). However by Proposition \ref{prop:periodic} we can choose $0\le \gamma < 2q^2$, so there exists a $0\le z_0 < q$ such that $|2qz_0-\gamma'| \le q$. Moreover, the number $|2qz_0-\gamma'|$ can be made strictly negative as soon as $\gamma'$ is not a multiple of $q$. Therefore, if $E$ is nef, then $\gamma'$ is a multiple of $q$; however $(q,\gamma')=1$ by assumption, so $q=1$, hence $\beta=2$.\\
If $k=9$, imposing that $E$ has a nonnegative intersection with all the effective $(-2)$-divisors $C$ with $FC \le 2$ forces similarly $\beta=3$, and it is immediate to check that all primitive divisors $E$ with $E^2=0$ and $EF=3$ are actually nef (using Proposition \ref{prop:periodic} we can restrict to $\gamma\in \{1,2\}$).\\
Finally, if $k=25$, the reasoning is analogous considering all effective $(-2)$-divisors $C$ with $FC \le 4$.
\end{proof}

\section{K3 surfaces of Picard rank $4\le \rho(X) \le 10$}

Let $X$ be an elliptic K3 surface of Picard rank $4 \le \rho(X) \le 10$, and assume the existence of an elliptic fibration $|F|$ on $X$ of maximal rank. Then $\NS(X)=U\oplus L$, where $L$ has no roots. In order to single out the N\'eron-Severi lattices of K3 surfaces of zero entropy, we want to apply Proposition \ref{prop:rho10} and Theorem \ref{teo:extension}. We will proceed inductively: we already have a complete list of lattices of rank $3$ of zero entropy, and Theorem \ref{teo:extension} allows us to obtain informations on the entropy of N\'eron-Severi lattices of higher rank. Recall that any even hyperbolic lattice of rank at most $10$ embeds in the K3 lattice (cf. \cite{nikulin9}, Theorem 1.14.4), hence the orthogonal complement $L$ of $U$ above can be any even negative definite lattice of rank $\rk(L)\le 8$.\newline

We start with the case $\rho(X)=4$; the procedure for higher ranks will be the same, but we will need the help of a computer. Consider an elliptic K3 surface $X$ with $\NS(X)=U\oplus L$, where
$$L=\begin{pmatrix}
-2k_1 &a\\
a &-2k_2
\end{pmatrix}$$
is a rank $2$, even, negative definite lattice. Since $L$ has no roots, we can assume that $2 \le k_1 \le k_2$; then, up to isometry of $L$, we can also assume that $|a|\le k_1$, and $a =k_1$ if $|a|=k_1$. Theorem \ref{teo:extension} reads:

\begin{teo} \label{teo:extension4}
Let $\NS(X)=U \oplus L$ as above, and assume that there exists $k \notin \{2,3,4,5,7,9,13,25\}$ such that $\langle -2k \rangle$ embeds primitively in $L$ and $|\det(L)|\ge 4k$. Then $X$ has positive entropy. In particular, if $k_1,k_2 \notin \{2,3,4,5,7,9,13,25\}$, then $X$ has positive entropy.
\end{teo}
\begin{proof}
The first part is an immediate consequence of Theorem \ref{teo:extension}. For the second part, just notice that
$$\det(L)=4k_1k_2-a^2 \ge 3k_1k_2 \ge 6k_2\ge 4k_2,$$
hence we conclude using the first part.
\end{proof}

\begin{oss}
If we remove the condition $|\det(L)|\ge 4k$ above, then the theorem does not hold anymore. Consider for instance two K3 surfaces $X_1,X_2$ such that $\NS(X_i)=U \oplus L_i$, with
$$L_1=\begin{pmatrix}
-4 &0\\
0 &-4
\end{pmatrix}, \qquad L_2=\begin{pmatrix}
-6 &0\\
0 &-6
\end{pmatrix}.$$
Then $\langle -12 \rangle \hookrightarrow L_2$ and $|\det(L_2)|=36 \ge 24$, therefore the elliptic fibration of maximal rank on $U\oplus \langle -12 \rangle$ extends to an elliptic fibration of maximal rank on $U \oplus L_2$. On the contrary, $\langle -20 \rangle \hookrightarrow L_1$, but $|\det(L_1)|=16 < 40$, and indeed the elliptic fibration of maximal rank on $U\oplus \langle -20 \rangle$ does not extend to an elliptic fibration of maximal rank on $U \oplus L_1$. In fact we will see that $X_1$ has zero entropy.
\end{oss}

Table \ref{table:4} lists all the lattices $L$ unique in their genus (cf. the list in \cite{nebe}) not satisfying the condition in Theorem \ref{teo:extension4}:

\begin{table}[h]
\centering
\begin{tabular}{c|c|c|c} 
    $i$ &$-2k_1$ &$-2k_2$ &$a$\\
    \hline
$1$    &$-14$ &$-6$ &$3$\\
$2$    &$-10$ &$-4$ &$2$\\
$3$    &$-10$ &$-4$ &$0$\\
$4$    &$-8$ &$-6$ &$0$\\
$5$    &$-8$ &$-4$ &$2$\\
$6$    &$-6$ &$-6$ &$3$\\
\end{tabular}
\qquad
\begin{tabular}{c|c|c|c} 
    $i$ &$-2k_1$ &$-2k_2$ &$a$\\
    \hline

$7$    &$-6$ &$-6$ &$1$\\
$8$    &$-6$ &$-4$ &$2$\\
$9$    &$-6$ &$-4$ &$0$\\
$10$    &$-4$ &$-4$ &$2$\\
$11$    &$-4$ &$-4$ &$1$\\
$12$    &$-4$ &$-4$ &$0$\\
\end{tabular}

\caption{List of lattices of rank $4$ unique in their genus not satisfying the condition in Theorem \ref{teo:extension4}.}
\label{table:4}
\end{table}

All these $12$ lattices satisfy the condition that $\Or(L) \rightarrow \Or(A_L)$ is surjective. Therefore:

\begin{prop} \label{prop:counterex}
Let $X$ be a K3 surface, $\NS(X)=U\oplus L$, where $L$ is one of the $12$ rank $4$ lattices listed above. Then $X$ has a unique elliptic fibration up to automorphism.
\end{prop}

\begin{prop} \label{prop:pos4}
The two lattices $U\oplus L_i$, where $i \in \{4,5\}$ in the list above, have positive entropy.
\end{prop}
\begin{proof}
We show that both N\'eron-Severi lattices admit a second elliptic fibration, isomorphic to the original one by the previous proposition. For $\NS(X)=U\oplus L_4$ consider
$$E=[25,12,6,2].$$
$E$ is primitive, $E^2=0$ and $ES_0=1$, where $S_0$ is the zero section of $F$. Applying Lemma \ref{lemma:int}, $E$ is nef (and inducing an elliptic fibration of maximal rank) if and only if, for all effective $(-2)$-divisors $[x,y,z_1,z_2]\in \NS(X)$, we have
$$-\frac{1}{2} \|v\|_{L_4} - 12^2 = 4(6y-12z_1)^2+3(2y-12z_2)^2-12^2=12[12(y-2z_1)^2+(y-6z_2)^2-12] > 0.$$
If $[x,y,z_1,z_2]\in \NS(X)$ is an effective $(-2)$-curve not satisfying this inequality, then necessarily $y=2z_1$ and $(y-6z_2)^2 < 12$, so $(z_1-3z_2)^2 < 3$. Recall that by Lemma \ref{lemma:int} we have that
$$y \mid -\frac{1}{2} \|z\|_{L_4} -1 = 4z_1^2+3z_2^2-1.$$
From above we have $z_1-3z_2=0,\pm 1$. However, it is rather straightforward to check that all three possibilities cannot happen because of the divisibility above.\\
For $\NS(X)=U\oplus L_5$ the reasoning is analogous, considering $E=[15,7,4,2]$.
\end{proof}

Therefore it only remains a list of $10$ candidate N\'eron-Severi lattices of rank $4$ and zero entropy. We postpone the proof that all these $10$ N\'eron-Severi lattices actually have zero entropy to the end of the section; first, we want to find a similar list of candidate lattices of rank $5 \le \rho(X) \le 10$. The algorithms described here below can be found in the ancillary folder \texttt{Find\_list\_candidates}.\newline

Recall that, since $\rho(X)\le 10$, we only have to worry about the N\'eron-Severi lattices decomposing as $U\oplus L$, with $L$ unique in its genus (cf. Proposition \ref{prop:rho10}). Equivalently, $L$ must be a multiple of a lattice in Watson's list (cf. Theorem \ref{teo:watson}). If we can bound these multiples, we would only have to deal with a finite number of lattices. The idea of the following algorithm is to use Theorem \ref{teo:extension}; for, fix a lattice $L$ of rank $n$ in Watson's list and choose any primitive sublattice $L'\subseteq L$ of corank $1$. Then a high multiple of $L$ will satisfy the two assumptions of Theorem \ref{teo:extension}: the finiteness of the list of candidate lattices in rank $n-1$ implies that $U\oplus L'(m)$ will have positive entropy for $m\gg 0$, and $|\det(L(m))|\ge 2 |\det(L'(m))|$ for $m\gg 0$.

\begin{alg} \label{alg:lista}
Fix $2\le n \le 7$. Let $\Lc_n$ be the finite list of candidate lattices of rank $n$ (this list is finite by an inductive argument, since we have such a list for $n=2$). If $L \in \Lc_n$, we define $b(L)$ to be the greatest integer $b \ge 1$ such that $\frac{1}{b} L$ is still an even integral lattice. We put $b_n:=\max_{L\in \Lc_n}{b(L)}<\infty$.\\
Now let $L$ be an even lattice of rank $n+1$ with $b(L)=1$ and unique in its genus, and consider its first rank $n$ principal minor $L'$ (i.e. choose any basis of $L$ and let $L'$ be the primitive sublattice generated by the first $n$ elements of the basis). Let $c(L)$ be the smallest integer greater or equal than $2\frac{|\det(L')|}{|\det(L)|}$, and $d_n(L)=\max\{b_n,c(L)\}<\infty$.\\
Consider the finite (by Theorem \ref{teo:watson}) list
$$\Lc'_{n+1}:=\{ L(m) \mid L \text{ even unique in its genus},\  b(L)=1, \  \rk(L)=n+1,\ m\in [1,d_n(L)]\}.$$
If $L\notin \Lc'_{n+1}$, then $L$ has positive entropy: indeed, either it is not unique in its genus, or $L= N(m)$ for some $N\in \Lc_{n+1}'$, $m > d_n(N)\ge b_n$. But in this case, by construction of the $b_n$, the first rank $n$ principal minor $L'$ of $L$ is not in $\Lc_n$, thus $U\oplus L'$ has positive entropy. Moreover,
$$\frac{|\det(L)|}{|\det(L')|}= m \frac{|\det(N)|}{|\det(N')|} \ge d_n(N) \frac{|\det(N)|}{|\det(N')|}\ge 2,$$
and we conclude by using Theorem \ref{teo:extension}.\\
We remove from $\Lc'_{n+1}$ all lattices $L$ such that $\min(L)=2$. Now, for every $L\in \Lc'_{n+1}$, we consider various immersions $L'\hookrightarrow L$, where $\rk(L')=n$ (first, we consider the $n+1$ principal minors of $L$, then other sufficiently many random primitive rank $n$ sublattices) and we check if $L'\notin \Lc_n$ and $|\det(L)|\ge 2|\det(L')|$. If both conditions hold, we remove $L$ from $\Lc'_{n+1}$. At the end, we return $\Lc_{n+1}:=\Lc'_{n+1}$.
\end{alg}

\begin{oss}
Checking if $L'\notin \Lc_n$ is computationally very fast, since we have only to check if the genus of $L'$ coincides with the genus of some lattice in $\Lc_n$ (recall that all lattices in $\Lc_n$ are unique in their genus).
\end{oss}

\begin{oss} \label{oss:alg}
Algorithm \ref{alg:lista} works similarly also for $n\ge 8$, changing the condition $|\det(L)|\ge 2|\det(L')|$ with the more restrictive $|\det(L)|> 2|\det(L')|$ (and defining $c(L)$ as the smallest integer \textit{strictly} greater than $2\frac{|\det(L')|}{|\det(L)|}$).
\end{oss}

To complete the classification at each step we search the candidate lattices for possible new elliptic fibrations, just as we did in Proposition \ref{prop:pos4}. The following lemma contains the algorithm that we will use to check whether a primitive element $E\in \NS(X)$ with $E^2=0$ is nef.

\begin{lemma} 
Let $X$ be an elliptic K3 surface with $\NS(X)=U \oplus L$ and $L$ without roots. Let $E=[\alpha,\beta,\gamma]\in \NS(X)$ be a primitive element with $E^2=0$. To any $v\in L$ we associate the finite set
$$I(v)=\left\{ y\in \N : y \mid -\frac{1}{2}\|v\|_L-\beta^2, \  z=\frac{1}{\beta}(y\gamma-v)\in L \text{ and } y \mid -\frac{1}{2}\|z\|_L-1\right\},$$
where $z\in L$ means that $z$ has integer entries. Then $E$ is nef if and only if
$$I(v)=\emptyset \text{ for all } v \in L \text{ with } -\frac{1}{2}\|v\|_L < \beta^2.$$
\end{lemma}
\begin{proof}
$E$ is not nef if and only if there exists $C=[x_0,y_0,z_0]$ with $C^2=-2$, $y_0>0$ such that $EC<0$. Let $v=y_0\gamma-\beta z_0$. Then Lemma \ref{lemma:int} shows that $y_0 \mid -\frac{1}{2}\|z_0\|_L-1$ and $-\frac{1}{2}\|v\|_L < \beta^2$. Moreover
$$-\frac{1}{2}\|v\|_L-\beta^2 = -\frac{1}{2}\|y_0\gamma-\beta z_0\|_L-\beta^2 \equiv -\frac{1}{2}\|-\beta z_0\|_L-\beta^2 = \beta^2\left(-\frac{1}{2}\|z_0\|_L-1\right)\equiv 0 \pmod{y_0},$$
hence $y_0 \in I(v)$ and thus $I(v)\ne \emptyset$. Conversely, assume that $y_0 \in I(v)$ for a $v\in L$ with $-\frac{1}{2}\|v\|_L < \beta^2$. Put $z_0=\frac{1}{\beta}(y_0\gamma-v)\in L$ and choose $x_0$ such that $C=[x_0,y_0,z_0]$ has $C^2=-2$ (this is possible since $y_0 \mid -\frac{1}{2}\|z_0\|_L-1$). Then Lemma \ref{lemma:int} shows that $EC<0$, hence $E$ is not nef.
\end{proof}

\begin{oss}
\begin{itemize}
    \item This lemma gives a practical way to decide whether a primitive divisor of square zero is nef. Indeed, the set of $v\in L$ satisfying $\frac{1}{2}\|v\|_L < \beta^2$ is finite, since $L$ is negative definite, so we only have to perform a finite number of checks.
    \item The lemma can be generalized to any $L$. Let $L$ be any even negative definite lattice, and $E\in U\oplus L$ primitive of square zero. Consider the root part $R=L_{root}\subseteq L$, and say that $R$ is generated by effective roots $r_1,\ldots,r_m$. Then the effective roots (i.e. effective divisors of square $-2$) in $U\oplus L$ can be orthogonal or not to the given elliptic fiber $F=[1,0,0]\in \NS(X)$. If $r$ is an effective root with $rF=0$, then $r$ is a linear combination of $r_1,\ldots,r_m$ with nonnegative coefficients. If instead $rF>0$, then $r=[x,y,z]\in\NS(X)$ has $y>0$, and hence we can apply the previous lemma. Summing up, we obtain that $E$ is nef if and only if the sets $I(v)$ as in the lemma are empty, and $Er_i\ge 0$ for all $i=1,\ldots,m$.
    \item This lemma is a result analogous to Proposition 4.1 in \cite{shimada2}. Shimada's algorithm checks the nefness of a divisor of positive square, while ours checks it for elements of square $0$. Both algorithms boil down to listing some short vectors in $L=U^\bot \subseteq \NS(X)$.
\end{itemize}

\end{oss}

\begin{cor} \label{cor:alg}
Let $X$ be an elliptic K3 surface with $\NS(X)=U \oplus L$ and $L$ without roots. Let $F=[1,0,0]\in \NS(X)$ be the given elliptic fibration on $X$. Checking whether there exists a primitive nef $E\in \NS(X)$ such that $E^2=0$, $EF=\beta$ is a computationally finite problem for any $\beta \ge 2$.
\end{cor}
\begin{proof}
Proposition \ref{prop:periodic} shows that without loss of generality we can consider $E=[\alpha,\beta,\gamma]$ with all the entries of $\gamma$ in the interval $(-\beta,\beta]$. Since $\alpha$ is uniquely determined by $\beta,\gamma$, this gives only a finite number of such divisors $E$. We just apply the previous lemma to each of them.
\end{proof}

The previous lemma allows us to search for elliptic fibrations on our candidate lattices. Notice that each list $\Lc_n$ in Algorithm \ref{alg:lista} contains only lattices $L$ unique in their genus, so checking whether an elliptic curve $E$ on $L$ induces an elliptic fibration with infinitely many sections coincides with checking whether the fibration $|E|$ has \textit{at least} a section.

\begin{alg} \label{alg:ell}
Let $\Lc_n$ be the list of lattices obtained in Algorithm \ref{alg:lista} and choose $L\in \Lc_n$. Pick a divisor $\beta>1$ of $\det(L)$, and search for nef primitive divisors $E\in \NS(X)$ with $E^2=0$ and $FE=\beta$, as explained in Corollary \ref{cor:alg}. For all such divisors, we search for sections $[x,y,z]$ with ``sufficiently small'' $y,z$. As soon as we find such a divisor with a section, we stop the algorithm and we remove $L$ from $\Lc_n$.
\end{alg}

Now we are ready to run the two algorithms \ref{alg:lista}, \ref{alg:ell}, obtaining the following:

\begin{teo} \label{thm:list0}
The candidate N\'eron-Severi lattices of K3 surfaces $X$ of Picard rank $\rho(X) \le 10$ admitting an elliptic fibration with only irreducible fibers and zero entropy are of the form $U\oplus L$, where $L$ is isomorphic to one of the following $32$ lattices sorted by rank:
$$1: \begin{pmatrix}
-4
\end{pmatrix},
\begin{pmatrix}
-6
\end{pmatrix},
\begin{pmatrix}
-8
\end{pmatrix},
\begin{pmatrix}
-10
\end{pmatrix},
\begin{pmatrix}
-14
\end{pmatrix},
\begin{pmatrix}
-18
\end{pmatrix},
\begin{pmatrix}
-26
\end{pmatrix},
\begin{pmatrix}
-50
\end{pmatrix}$$
$$2:\begin{pmatrix}
-14 &3\\
3 &-6
\end{pmatrix},
\begin{pmatrix}
-10 &2\\
2 &-4
\end{pmatrix},
\begin{pmatrix}
-10 &0\\
0 &-4
\end{pmatrix},
\begin{pmatrix}
-6 &3\\
3 &-6
\end{pmatrix},
\begin{pmatrix}
-6 &1\\
1 &-6
\end{pmatrix},$$
$$\begin{pmatrix}
-6 &2\\
2 &-4
\end{pmatrix},
\begin{pmatrix}
-6 &0\\
0 &-4
\end{pmatrix},
\begin{pmatrix}
-4 &2\\
2 &-4
\end{pmatrix},
\begin{pmatrix}
-4 &1\\
1 &-4
\end{pmatrix},
\begin{pmatrix}
-4 &0\\
0 &-4
\end{pmatrix}$$
$$3:\begin{pmatrix}
-4 &-2 &-2\\
-2 &-4 &-2\\
-2 &-2 &-6
\end{pmatrix},
\begin{pmatrix}
-4 &-1 &-1\\
-1 &-4 &1\\
-1 &1 &-4
\end{pmatrix},
\begin{pmatrix}
-4 &2 &2\\
2 &-6 &-1\\
2 &-1 &-6
\end{pmatrix},
\begin{pmatrix}
-4 &1 &2\\
1 &-4 &1\\
2 &1 &-4
\end{pmatrix},$$
$$\begin{pmatrix}
-4 &1 &1\\
1 &-4 &-1\\
1 &-1 &-4
\end{pmatrix},
\begin{pmatrix}
-4 &2 &0\\
2 &-4 &0\\
0 &0 &-6
\end{pmatrix},
\begin{pmatrix}
-4 &-2 &2\\
-2 &-4 &0\\
2 &0 &-4
\end{pmatrix}$$
$$4:\begin{pmatrix}
-4 &0 &0 &-2\\
0 &-4 &0 &-2\\
0 &0 &-4 &-2\\
-2 &-2 &-2 &-4
\end{pmatrix},
\begin{pmatrix}
-4 &-2 &-1 &1\\
-2 &-4 &1 &-1\\
-1 &1 &-4 &1\\
1 &-1 &1 &-4
\end{pmatrix},
\begin{pmatrix}
-4 &1 &1 &1\\
1 &-4 &1 &1\\
1 &1 &-4 &1\\
1 &1 &1 &-4
\end{pmatrix},
\begin{pmatrix}
-4 &-1 &-2 &2\\
-1 &-4 &1 &-1\\
-2 &1 &-4 &1\\
2 &-1 &1 &-4
\end{pmatrix}$$
$$5:\begin{pmatrix}
-4 &-1 &-1 &-1 &-2\\
-1 &-4 &-1 &-1 &-2\\
-1 &-1 &-4 &-1 &-2\\
-1 &-1 &-1 &-4 &1\\
-2 &-2 &-2 &1 &-4
\end{pmatrix}$$
$$6:\begin{pmatrix}
-4 &1 &-1 &-1 &0 &0\\
1 &-4 &-2 &1 &0 &0\\
-1 &-2 &-4 &2 &-3 &0\\
-1 &1 &2 &-4 &3 &0\\
0 &0 &-3 &3 &-6 &-3\\
0 &0 &0 &0 &-3 &-6
\end{pmatrix}$$
$$8:E_8(2).$$
\end{teo}

At this point, we have to prove that these remaining N\'eron-Severi lattices admit a unique elliptic fibration. The strategy will be the following: let $X$ be any of the previous K3 surfaces and $|F|$ the given elliptic fibration on $X$ of maximal rank. Then we prove that there exists a special subset $\Delta_X$ of effective $(-2)$-divisors with the property that, for any effective divisor $F\ne E$ on $X$ with $E^2=0$, there exists $C\in \Delta_X$ such that $EC\le 0$. This proves that no elliptic curve on $X$ can induce an elliptic fibration of maximal rank, and thus that $X$ has zero entropy. Moreover the only possible elliptic curves on $X$ (that consequently will have no sections) correspond to the divisors $E$ as above for which $EC\ge 0$ for all $C\in \Delta_X$. It turns out that in all the cases it suffices to choose $\Delta_X$ as the set of effective $(-2)$-curves $C$ such that the intersection $FC$ is bounded by a constant depending on $X$.

\begin{oss}
The K3 surfaces with Picard lattice isomorphic to $U\oplus E_8(2)$ were already studied in \cite{nikulin2} and proven to have zero entropy (cf. Theorems 4.2.2 and 4.2.4).
\end{oss}

\begin{teo} \label{teo:finale4}
The elliptic K3 surfaces $X$ such that $\NS(X)=U\oplus L$, with $L$ one of the previous $32$ lattices, have a unique elliptic fibration, hence zero entropy. Moreover the following table specifies whether such surfaces admit other elliptic pencils:

\begin{table}[H]
\centering
\begin{tabular}{c|c|c|c} 
    $\rho(X)$ &$\#$ &Other genus $1$ fibr. &$\beta=E F$\\
    \hline
\multirow{10}{1em}{$4$} &$1$ &Yes &$5$\\
&$2$ &Yes &$3$\\
&$3$ &No &$-$\\
&$4$ &Yes &$3$\\
&$5$ &No &$-$\\
&$6$ &No &$-$\\
&$7$ &No &$-$\\
&$8$ &No &$-$\\
&$9$ &No &$-$\\
&$10$ &Yes &$2$
\end{tabular}
\qquad
\begin{tabular}{c|c|c|c} 
    $\rho(X)$ &$\#$ &Other genus $1$ fibr. &$\beta=E F$\\
    \hline
\multirow{7}{1em}{$5$} &$1$ &No &$-$\\
&$2$ &No &$-$\\
&$3$ &Yes &$3$\\
&$4$ &No &$-$\\
&$5$ &Yes &$3$\\
&$6$ &Yes &$3$\\
&$7$ &Yes &$2$\\
\hline
\multirow{4}{1em}{$6$} &$1$ &Yes &$2$\\
&$2$ &Yes &$3$\\
&$3$ &Yes &$5$\\
&$4$ &Yes &$3$\\
\hline
$7$ &$1$ &Yes &$3$\\
\hline 
$8$ &$1$ &Yes &$3$\\
\hline 
$10$ &$1$ &Yes &$2$\\
\end{tabular}

\caption{Genus $1$ fibrations with no sections on K3 surfaces of zero entropy. The last column indicates the smallest intersection number of these elliptic curves with the fiber of the unique elliptic fibration.}
\label{tab:genus1}
\end{table}

\end{teo}
\begin{proof}
Let $L$ be one of the lattices above, $n=\rk(L)$. We want to prove that there exists a unique elliptic fibration on $X$, so let $E=[\alpha,\beta,\gamma]\in U\oplus L$ be primitive of square $0$ with $\beta>0$; the goal is to show that $E$ does not induce an elliptic fibration.\\
First, we can assume that the entries of $\gamma$ are in $[0,\beta-1]$ by Proposition \ref{prop:periodic}. If $E$ induces an elliptic fibration, then $|E|$ has no reducible fibers, since $L$ is unique in its genus and it has no roots. This means that $EC>0$ for any effective $(-2)$-divisor $C=[x,y,z]\in U\oplus L$. By Lemma \ref{lemma:int}, the inequality can be rewritten as
$$-\frac{1}{2}\|y\gamma-\beta z\|_L - \beta^2>0, \quad \text{or equivalently} \quad -\frac{1}{2}\left\|y\frac{\gamma}{\beta}- z\right\|_L >1.$$
The vector $c=\frac{\gamma}{\beta}$ has rational entries between $0$ and $1$. If we are able to find finitely many effective $(-2)$-divisors $C_i:=[x_i,y_i,z_i]$ such that the $n$-dimensional balls
$$B_i=\left\{-\frac{1}{2}\|y_ic-z_i\|_L \le 1\right\}$$
cover the hypercube $[0,1]^n$, then we are done. Indeed this shows that, for any primitive isotropic $E$ different from the given elliptic fiber $F$, $EC_i \le 0$ for at least one of our $(-2)$-divisors $C_i$. This implies that $E$ cannot induce an elliptic fibration.\\
We will explain one example in detail; the others are checked similarly with the help of a computer. We refer to the ancillary folder \texttt{Prove\_Zero\_Entropy} for the details, as the number of $(-2)$-divisors that we need to consider grows considerably with the rank of $L$. Let
$$L=\begin{pmatrix}
-6 &2\\
2 &-4
\end{pmatrix}.$$
We consider the effective $(-2)$-divisors
$$C_1=[0,1,0,0],\quad C_2=[3,1,1,0],\quad C_3=[2,1,0,1],\quad C_4=[3,1,1,1].$$
These yield the balls
$$B_1=\{3c_1^2-2c_1c_2+2c_2^2\le 1\},\quad B_2=\{3(c_1-1)^2-2(c_1-1)c_2+2c_2^2\le 1\},$$ 
$$B_3=\{3c_1^2-2c_1(c_2-1)+2(c_2-1)^2\le 1\}, \quad B_4=\{3(c_1-1)^2-2(c_1-1)(c_2-1)+2(c_2-1)^2\le 1\}.
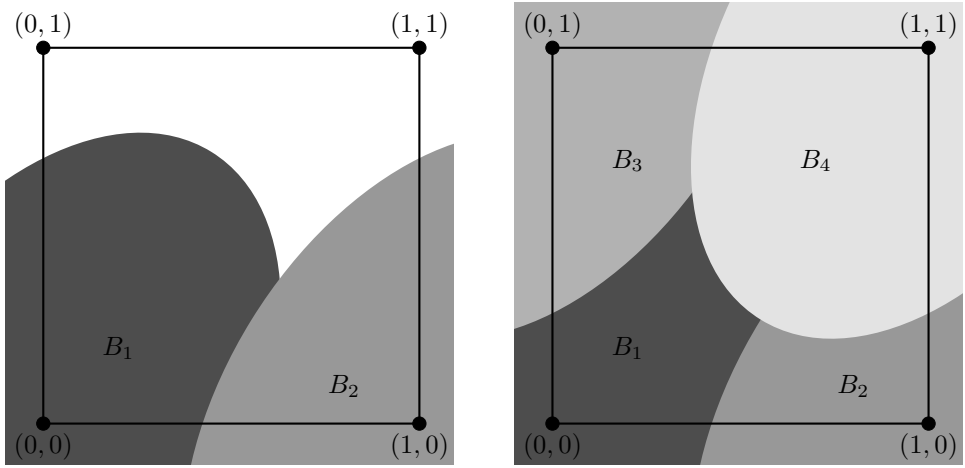
$$
\vspace{0.5 cm}

\begin{center}
\begin{tikzpicture}[scale=5]
\clip (-0.1,-0.11) rectangle  (1.09,1.12);

\fill[red, rotate=58.29] (0,0) ellipse (0.85 and 0.525);
\fill[orange, rotate around={58.29:(1,0)}] (1,0) ellipse (0.85 and 0.525);

\draw[black,thick] (0,0) -- (1,0);
\draw[black,thick] (1,0) -- (1,1);
\draw[black,thick] (1,1) -- (0,1);
\draw[black,thick] (0,0) -- (0,1);

\filldraw[black] (0,0) circle (0.5 pt);
\filldraw[black] (1,0) circle (0.5 pt);
\filldraw[black] (0,1) circle (0.5 pt);
\filldraw[black] (1,1) circle (0.5 pt);
\node[below] at (0,0) {$(0,0)$};
\node[below] at (1,0) {$(1,0)$};
\node[above] at (0,1) {$(0,1)$};
\node[above] at (1,1) {$(1,1)$};

\node at (0.2,0.2) {$B_1$};
\node at (0.8,0.1) {$B_2$};

\end{tikzpicture}
\qquad
\begin{tikzpicture}[scale=5]
\clip (-0.1,-0.11) rectangle  (1.09,1.12);

\fill[red, rotate=58.29] (0,0) ellipse (0.85 and 0.525);
\fill[orange, rotate around={58.29:(1,0)}] (1,0) ellipse (0.85 and 0.525);
\fill[cyan, rotate around={58.29:(0,1)}] (0,1) ellipse (0.85 and 0.525);
\fill[yellow, rotate around={58.29:(1,1)}] (1,1) ellipse (0.85 and 0.525);

\draw[black,thick] (0,0) -- (1,0);
\draw[black,thick] (1,0) -- (1,1);
\draw[black,thick] (1,1) -- (0,1);
\draw[black,thick] (0,0) -- (0,1);

\filldraw[black] (0,0) circle (0.5 pt);
\filldraw[black] (1,0) circle (0.5 pt);
\filldraw[black] (0,1) circle (0.5 pt);
\filldraw[black] (1,1) circle (0.5 pt);
\node[below] at (0,0) {$(0,0)$};
\node[below] at (1,0) {$(1,0)$};
\node[above] at (0,1) {$(0,1)$};
\node[above] at (1,1) {$(1,1)$};

\node at (0.2,0.2) {$B_1$};
\node at (0.8,0.1) {$B_2$};
\node at (0.2,0.7) {$B_3$};
\node at (0.7,0.7) {$B_4$};

\end{tikzpicture}
\end{center}

As we can see from the picture, the balls indeed cover the square $[0,1]^2$, as wanted.\\
For the second part of the statement we repeat a similar process. First, we run Algorithm \ref{alg:ell} (more precisely, a modified version of it that only checks the nefness of an isotropic divisor) for ``small'' values of $\beta$, and we find other genus $1$ fibrations on some of our lattices, as specified by Table \ref{tab:genus1}. For the remaining lattices, we implement the exact same strategy as in the first part of the proof. More precisely, assume that the open balls 
$$\mathring{B}_i=\left\{-\frac{1}{2}\|y_ic-z_i\|_L < 1\right\}$$
still cover the hypercube $[0,1]^n$. Then the corresponding N\'eron-Severi lattice does not admit any other genus $1$ fibration. Indeed, if $E=[\alpha,\beta,\gamma]$ different from $F$ is nef, primitive and isotropic, then $EC_i\ge 0$ for all our effective $(-2)$-divisors $C_i$, implying that the vector $c=\frac{\gamma}{\beta}\notin \mathring{B}_i$ for all $i$, a contradiction.

\end{proof}

\begin{question}
How many genus $1$ fibrations are there up to automorphisms on these zero entropy K3 surfaces?
\end{question}

\section{K3 surfaces of Picard rank $\rho(X)>10$}

Let $X$ be an elliptic K3 surface of Picard rank $\rho(X)>10$ admitting an elliptic fibration of maximal rank, $\NS(X)=U\oplus L$, $L$ without roots. As a consequence of Theorem \ref{teo:watson}, we have that all such $L$ are not unique in their genus, except for very few cases of rank $9\le \rk(L) \le 10$. These cases are easily worked out, using Algorithm \ref{alg:lista} combined with Remark \ref{oss:alg} and Algorithm \ref{alg:ell}:

\begin{cor}
Let $X$ be an elliptic K3 surface, $\NS(X)=U\oplus L$, $L$ one of the lattices without roots, with rank $\rk(L)>8$ and unique in their genus. Then $X$ has positive entropy.
\end{cor}

Therefore, from now on, we will always assume that $L$ is not unique in its genus. Equivalently, there exists a lattice $M$, not isometric to $L$, in the genus of $L$. If there exists such an $M$ which is not a root-overlattice, then $X$ has positive entropy. The whole section will focus on the proof of the following theorem:

\begin{teo} \label{teo:rho10}
Let $R$ be a root-overlattice of rank $9 \le r=\rk(R) \le 18$ satisfying the condition $\det(R) \ge \Delta_r$, where $\Delta_r$ can be found in Table \ref{table:bound}, and such that $U \oplus R$ embeds primitively in the K3 lattice. Then there exists a lattice $N$ in the genus of $R$ that is not a root-overlattice and such that $\min(N)=2$.
\end{teo}

This gives immediately:

\begin{cor} \label{cor:bello}
Let $X$ be an elliptic K3 surface admitting an elliptic fibration with only irreducible fibers, $\NS(X)=U\oplus L$. Assume that $L$ is not unique in its genus. Then $X$ admits at least two non-isomorphic elliptic fibrations with infinitely many sections, and in particular $X$ has positive entropy.
\end{cor}
\begin{proof}
Let $M$ be a lattice non isometric to $L$ in the genus of $L$. If $M$ is not a root-overlattice, we are done. If instead $M$ is a root-overlattice, then the rank $r$ of $L$ must be at least $9$ by Proposition \ref{prop:rootgenus}, and $\det(M)=\det(L)$ must be at least $\Delta_r$ by Theorem \ref{teo:deltar}. Hence the previous theorem shows the existence of another lattice $N$, in the genus of $M$ (and thus in the genus of $L$), not isometric to $L$ (since $\min(N)=2$), which is not a root-overlattice, concluding the proof.
\end{proof}

The proof of Theorem \ref{teo:rho10} will be performed using the software \texttt{Magma}. Notice that we only have to deal with a finite number of root lattices, and some of their overlattices. A priori, one might think of computing the whole genus of all such root-overlattices, but it is quite easy to understand that this is computationally not feasible. Therefore we implemented some restrictions to exclude most of the cases.\newline

We start considering root lattices, which will form our base case.

\begin{lemma}
Let $R_0$ be a root lattice admitting a lattice $N_0$ that is not a root-overlattice in its genus. Then, for any root lattice $R_1$ of positive rank, the root lattice $R=R_0 \oplus R_1$ admits a non-root-overlattice $N_1$ in its genus, such that $\min(N_1)=2$.
\end{lemma}
\begin{proof}
Just consider $N_1=N_0 \oplus R_1$.
\end{proof}

This lemma, despite being very easy, furnishes us a quick way to eliminate many root lattices of high rank. Therefore, we construct an algorithm that finds a list $\Rc$ of root lattices admitting a non-root-overlattice in the genus satisfying the following property: if $R$ is a root lattice satisfying the hypotheses of Theorem \ref{teo:rho10}, then there exist two root lattices of positive rank, $R_0,R_1$, such that $R_0\in \Rc$ and $R=R_0\oplus R_1$.

It is in general very difficult to understand whether $U\oplus R$ embeds primitively in the K3 lattice, but there are few necessary conditions in order for that to hold. We have implemented the following:
\begin{itemize}
    \item[C1.] If $e(X_t)$ denotes the Euler characteristic of the fiber of $X$ over $t\in \Prj^1$, then
    \begin{equation} \label{eq:24}
        24=e(X)=\sum_{t \in T}{e(X_t)}\ge \sum_{t \in T_{\text{red}}}{e(X_t)},
    \end{equation}
    where $T$ (resp. $T_{\text{red}}$) denotes the set of $t\in \Prj^1$ having singular (resp. reducible) fibers. Recall that the Euler characteristic of a reducible fiber corresponding to a singularity of type $A_n$ (resp. $D_n$ or $E_n$) is at least $n+1$ (resp. $n+2$) (cf. \cite{miranda1}, Lemma IV.3.2, IV.3.3)
\end{itemize}

\begin{alg} \label{alg:ADE}
Construct the list $\Rc_0$ of all root lattices $R$ of rank $9 \le r=\rk(R) \le 18 $ satisfying $\det(R)\ge \Delta_r$ and the condition C1 above. $\Rc_0$ is clearly a finite list. Pick the first $R\in \Rc_0$ (after ordering them first by rank and then randomly), and assume that $R$ is a direct sum of the lattices $A_n,D_n,E_n$ indexed by $I$. For any $J\subsetneq I$, consider the sublattice $R_J \subsetneq R$ indexed by $J$, and order such $R_J$'s by rank, obtaining a list ${R_1,\ldots,R_k}$. If $R_1$ has a non-root-overlattice in its genus, then add $R_1$ to $\Rc$. Otherwise, do the same for $R_2$, until you find such an $R_i$. If it doesn't exist, exit the algorithm with an error message. Now remove all $R\in \Rc_0$ for which there exists a root lattice $R'$ such that $R=R_i \oplus R'$. Continue in this fashion until $\Rc_0 = \emptyset$, and return $\Rc$.
\end{alg}

\begin{oss}
In order to decide whether the $R_J$'s have non-root-overlattices in their genus, we use \texttt{Magma}'s function \texttt{GenusRepresentatives}, listing the whole genus of $R_J$ (if the rank of $R_J$ is small), or a modified version of \texttt{Neighbors}, listing all $2$-Neighbors (or $3$-,$5$-Neighbors if all $2$-Neighbors are root-overlattices) of $R_J$ (we refer to \cite{kneser} for the definition of neighbors, and to \cite{elkies} for a geometric interpretation of neighbors). In any case, we store for each $R \in \Rc$ as many non-root-overlattices in the genus of $R$ as possible: we will use them in a subsequent algorithm.
\end{oss}

After running Algorithm \ref{alg:ADE}, we get a list $\Rc$ of $131$ root lattices, together with the corresponding list of non-root-overlattices in their genus. This shows that Theorem \ref{teo:rho10} holds if $R$ is a root lattice.\newline

Now we want to extend this result to all root-overlattices. We need a couple of preliminary, well-known facts about overlattices.

\begin{lemma} \label{lemma:_over}
\begin{enumerate}
    \item Let $L$ be any even, negative definite lattice, and consider two isotropic subgroups $S,S'< A_L$ such that there exists $\varphi\in \Or(L)$ with $\overline{\varphi}(S)=S'$. Then $S,S'$ give rise to isometric overlattices of $L$.
    \item Let $L,L'$ be two even, negative definite lattices in the same genus. Then, for every overlattice $P$ of $L$, there exists an overlattice $P'$ of $L'$ such that $P$ and $P'$ are in the same genus.
\end{enumerate}
\end{lemma}
\begin{proof}
\begin{enumerate}
    \item The two overlattices are obtained adjoining the generators of $S$ (more precisely, their preimages in $L^\vee$ under the projection $L^\vee \rightarrow A_L$) to $L$, so the isometry $\varphi$ of $L$ extends to an isometry of the two overlattices.
    \item $P$ corresponds to an isotropic subgroup $S<A_L$, which in turn can be seen as an isotropic subgroup $S'$ of $A_{L'}$ using the isometry $A_L \cong A_{L'}$. The overlattice $P'$ of $L'$ corresponding to $S'$ is then in the genus of $P$, since they have isometric discriminant groups.
\end{enumerate}
\end{proof}

We want to understand which overlattices of root lattices we have to consider. Recall that, if $R'$ is a root-overlattice with $R'_{root}=R$ such that $U\oplus R'$ embeds primitively in the K3 lattice, then the quotient $R'/R$, which is isomorphic to the isotropic subgroup $S<A_R$ corresponding to $R'$, is also isomorphic to the (finite) Mordell-Weil group of the elliptic fibration on $U\oplus R'$. On K3 surfaces, this can only be one of the following $12$ groups (cf. \cite{MP}, Table 4.5):
$$\Z/2\Z,\Z/3\Z,\Z/4\Z,\Z/2\Z\times \Z/2\Z, \Z/5\Z,\Z/6\Z,$$
$$\Z/7\Z,\Z/8\Z,\Z/2\Z\times \Z/4\Z,\Z/3\Z\times \Z/3\Z, \Z/2\Z\times \Z/6\Z, \Z/4\Z\times \Z/4\Z.$$
Now let $R$ be a root lattice satisying the condition C1 in equation (\ref{eq:24}), and $R'$ an overlattice of $R$ of index $k$ with $R'_{root}=R$, such that $U\oplus R'$ embeds primitively in the K3 lattice. Obviously $k^2 \mid \det(R)$, but we also have the following condition:

\begin{itemize}
    \item[C2.] If $R$ contains at least one $D_n$ or $E_n$ as a summand, then $k$ divides the greatest common divisor of the determinants of all the $D_n,E_n$ summands in $R$. In particular $k \le 4$.
\end{itemize}

This follows from the well-known fact that the restriction map
$$\TMW(X) \longrightarrow \Tors(X_t),$$
sending any torsion section on $X$ to its intersection point with the fiber $X_t$, is injective for all $t\in \Prj^1$, and that the number of torsion points on reducible fibers with additive reduction coincides with the determinant of the root lattice corresponding to its singularity (see for instance \cite{miranda1}, Corollary VII.3.3 and Lemma VII.3.5).\newline

Now we are ready to explain the algorithms we have used to prove Theorem \ref{teo:rho10}; we start with the one that computes the overlattices of a given root lattice.

\begin{alg} \label{alg:overlattices}
Let $R$ be any root lattice, and $S$ one of the $R^\vee/R$ groups listed above. Assume that $S$ is cyclic of order $n$. Then we search for isotropic elements $s\in A_R$ of order precisely $n$, and we consider them up to the action of $\Or(R)$ on $A_R$ (by part (1) of Lemma \ref{lemma:_over}). Moreover, we discard those $s$, whose normalized preimage in $R^\vee$ has norm $-2$: indeed, these elements would give rise to an overlattice of $R$ isometric to some root lattice. Consider the list of $s \in A_R$ up to $\Or(R)$ satisfying these conditions. We then construct the list $\mathcal{T}$ of corresponding overlattices of $R$; finally, we discard an $R'\in \mathcal{T}$ if there exists an $R'\ne R'' \in \mathcal{T}$ such that $R'$ and $R''$ are in the same genus. We return $\mathcal{T}$.\\
If instead $S$ has two generators, of orders say $n,m$, we search for isotropic elements $s\in A_R$ of order precisely $n$, and we consider them up to the action of $\Or(R)$ on $A_R$. Then, for each such $s$, we look for isotropic elements $s'\in A_R$ of order $m$ such that $s\cdot s'=0$, thus obtaining a list of pairs generating subgroups of $A_R$ isomorphic to $S$. We remove the pairs containing elements of norm $-2$ (as these correspond to overlattices of $R$ with a root part strictly bigger than $R$), and we conclude as in the previous case.
\end{alg}

We return $\mathcal{T}$ without repeating lattices in the same genus, since if Theorem \ref{teo:rho10} holds for one of them, it also holds for all the others in the same genus. The next is the main algorithm of the section:

\begin{alg} \label{alg:main10}
Construct the list $\Rc_0$ of all root lattices $R$ of rank $9 \le r=\rk(R) \le 18 $ satisfying $\det(R)\ge \Delta_r$ and the condition C1 above. Notice that the condition $\det(R)\ge \Delta_r$ is necessary, since we want the overlattices of $R$ to satisy that inequality. Pick $R\in \Rc_0$, and compute the list of finite groups $S$ above that can appear as quotients $R'/R$, with $R'$ a root-overlattice with $R'_{root}=R$. More precisely, we want $k=\#S$ to satisfy the conditions $k^2 \mid \det(R)$, $\det(R)/k^2 \ge \Delta_r$, and the condition C2 above. Now, for all such $S$, we compute all the root-overlattices $R'$ with $R'_{root}=R$ and $R'/R\cong S$, using Algorithm \ref{alg:overlattices}, and we choose one of them, say $R'$. Now we go through the list $\Rc$ we have constructed in Algorithm \ref{alg:ADE}, and select the lattices $R_0\in \Rc$ such that there exists a root lattice $R_1$ with $R=R_0\oplus R_1$ (there exists at least one such $R_0$, as Algorithm \ref{alg:ADE} shows). For any such $R_0$, we get many non-root-overlattices in the genus of $R$ (using the lattices stored in Algorithm \ref{alg:ADE}), and we compute their overlattices in the genus of $R'$ (which exist by part (2) of Lemma \ref{lemma:_over}). If there exists such an overlattice which is not a root-overlattice and whose minimum is $2$, we are done for $R'$. Otherwise, we try to find a $p$-Neighbor (for $p=2,3$) with minimum $2$ that is not a root-overlattice, and we return an error message if this doesn't work too. We repeat this for all $R',S$ and $R$.
\end{alg}

\begin{oss}
In general, searching for $p$-Neighbors is computationally very slow, compared to checking whether there exists an overlattice of a given lattice that is not a root-overlattice. Algorithm \ref{alg:main10} finishes in a reasonable amount of time, since it only has to search for $p$-Neighbors in very few cases (less than $60$).
\end{oss}

\begin{oss}
When the rank of our root lattice $R$ in Algorithm \ref{alg:main10} is $18$, we use Shimada-Zhang's list of extremal singular K3 surfaces (\cite{SZ}, Table 2) to decide whether a certain finite group $S$ can be the quotient $R'/R$ for some overlattice $R'$ of $R$. This simplifies the task of the algorithm.
\end{oss}

Algorithm \ref{alg:main10} terminates without any error message, thus finally proving Theorem \ref{teo:rho10}.\newline

Let us comment on this result. We have just proved that, if $X$ is a K3 surface admitting an elliptic fibration $|F|$ of maximal rank and a second elliptic fibration $|F_2|$ with finitely many sections, then $X$ admits a third ``intermediate'' elliptic fibration $|F_3|$, i.e. having $0<\rk(\MW(F_3))<\rho(X)-2$. This follows from the proof of Corollary \ref{cor:bello}. Interestingly, the existence of this third ``intermediate'' elliptic fibration heavily depends on the fact that $|F|$ has maximal rank, as the next example shows:

\begin{esempio}
Consider $R=A_1^9$. The lattice $U\oplus R$ embeds into the K3 lattice by \cite{morrison}, Remark 2.11. The genus of $R$ consists of $R$ itself and the lattice $L=A_1\oplus E_8(2)$, which has only one root. However, there are no ``intermediate'' lattices between $R$ and $L$ in the genus.
\end{esempio}

Moreover we want to point out that there actually exist K3 surfaces $X$ admitting both an elliptic fibration of maximal rank and an elliptic fibration with finitely many sections; we don't know how small the Picard rank $\rho(X)>10$ can be, but we know the following example in Picard rank $\rho(X)=20$:

\begin{esempio}
Consider the (unique up to isomorphism) singular K3 surface $X$ with transcendental lattice
$$\T(X)=\begin{pmatrix}
20 &10\\
10 &20
\end{pmatrix}$$
(cf. \cite{SI}, Theorem 4). Shioda in \cite{shioda1} gives an explicit Weierstrass equation for $X$, namely
$$y^2=x^3+t^5-\frac{1}{t^5}-11,$$
and he proves that the Mordell-Weil group of $X$ over $\Prj^1_t$ has maximal rank $18$. However, $X$ appears in Shimada-Zhang's list (\cite{SZ}, Table 2): in particular they show that $X$ admits an elliptic fibration for which $\NS(X)=U\oplus R$, with $R$ an overlattice of index $2$ of $A_1^3\oplus A_2 \oplus A_4\oplus A_9$.
\end{esempio}

Summing up all the results of the last three sections, we have:

\begin{teo} \label{teo:grosso}
Let $X$ be a K3 surface with an infinite automorphism group. Suppose that $X$ admits an elliptic fibration with only irreducible fibers. Then:
\begin{enumerate}
    \item $X$ has zero entropy, or equivalently $X$ admits a unique elliptic fibration with infinitely many sections, if and only if $\NS(X)$ belongs to an explicit list of $32$ lattices. In particular $\rho(X)\le 10$.
    \item $X$ admits a unique genus $1$ fibration if and only if $\NS(X)$ belongs to an explicit list of $14$ lattices. In particular $\rho(X) \le 5$.
\end{enumerate}
\end{teo}

\section{K3 surfaces of Picard rank $\ge 19$}

This last section is devoted to the proof of the following theorem:

\begin{teo} \label{teo:singular}
All K3 surfaces with Picard rank $\ge 19$ and infinite automorphism group have positive entropy.
\end{teo}

When the K3 surface $X$ is \textit{singular}, i.e. it has $\rho(X)=20$, it has been proven by Oguiso (\cite{oguiso1}, Theorem 1.6) that $X$ has positive entropy. Using the methods introduced earlier in the paper, we are able to extend his result to $\rho(X)=19$.\newline

Let $X$ be a K3 surface with Picard rank $19$. $X$ is elliptic by \cite{huybrechts1}, Corollary 14.3.8, so its N\'eron-Severi lattice is $\NS(X)=U\oplus L'$, for a certain negative definite lattice $L'$ of rank $17$. The trascendental lattice $\T(X)=\NS(X)^\bot$ has rank $3$ and signature $(2,1)$, so it embeds into the unimodular lattice $U^2\oplus E_8$ by \cite{nikulin9}, Corollary 1.12.3. This implies that $\NS(X)$ contains at least a copy of $E_8$, hence $\NS(X)=U\oplus E_8 \oplus L$ for a certain negative definite lattice $L$ of rank $9$.

\begin{oss} \label{oss:ultim}
From \cite{nikulin10} we know that the automorphism group of $X$ is finite if and only if $\NS(X)\cong U\oplus E_8\oplus E_8\oplus A_1$. In all the other cases $X$ admits an elliptic fibration with infinitely many sections.
\end{oss}

\begin{teo}
Let $X$ be a K3 surface with $\rho(X)=19$ and an infinite automorphism group. Then $X$ admits at least two distinct elliptic fibrations with infinitely many sections. Equivalently, $X$ has positive entropy.
\end{teo}
\begin{proof}
Let $\NS(X)=U\oplus E_8\oplus L$. We first consider the genus of $L$. Indeed, if the genus of $L$ contains at least two non-isometric non-root-overlattices, then any K3 surface $Y$ with $\NS(Y)=U\oplus L$ has positive entropy, and it admits two distinct elliptic fibrations with infinitely many sections. If these two elliptic fibrations are induced by $E_1,E_2\in U \oplus L$, it is easy to notice that the extensions $[E_1,0],[E_2,0]\in U\oplus L\oplus E_8$ induce distinct elliptic fibrations with infinitely many sections on $X$, thus $X$ has positive entropy.

Assume that $L$ is unique in its genus. Then \cite{nebe} shows that $L$ is a multiple of one of $4$ lattices: $L_1=E_8\oplus A_1$, $L_2=E_8(4)\oplus A_1$, $L_3,L_4$, where $L_3$ has no roots and $L_4$ has $\rk((L_4)_{root})=8$. Theorem \ref{teo:grosso} proves that $U\oplus L$ has positive entropy whenever $L$ is $L_2,L_3$, or any multiple $L_1(m),L_2(m),L_3(m),L_4(m)$ with $m>1$. As above, if $U\oplus L$ has positive entropy, then also $U\oplus L\oplus E_8$ has positive entropy. Moreover, by Remark \ref{oss:ultim} we can discard $L_1$, as $U\oplus E_8\oplus L_1$ has a finite automorphism group. We thus only have to consider $L=L_4$; we will deal with it at the end of the proof.

Assume instead that $L$ is not unique in its genus. If the genus of $L$ contains no root-overlattices, then $U\oplus L$ has positive entropy (and therefore $X$ has positive entropy) by Theorem \ref{teo:unique2}. Hence we can assume that $L$ is a root-overlattice. We can easily list all root-overlattices of rank $9$ using Algorithm \ref{alg:overlattices}, obtaining $53$ distinct genera of root-overlattices. Studying these genera with \texttt{Magma}, we find out that $41$ of these $53$ contain at least two non-isometric non-root-overlattices (cf. the ancillary file \texttt{NonRootOverlattices9}), hence by the remark at the beginning of the proof these give rise to K3 surfaces of positive entropy. After again discarding the genus of $E_8\oplus A_1$, we remain with the genera of the following $12$ lattices:
$$A_1^9,D_4\oplus A_1^5,D_4^2\oplus A_1,D_4\oplus D_5,E_6\oplus A_2\oplus A_1,D_6\oplus A_3,E_7\oplus A_1^2,D_7\oplus A_2,E_7\oplus A_2,D_9,O,L_4,$$
where $L_4$ is the lattice unique in its genus discussed above, and $O$ is an overlattice of $A_3\oplus A_1^6$ of index $2$.

It is easy to check that the genera of $D_4\oplus D_5,D_6\oplus A_3,E_7\oplus A_2,D_9,O$ contain respectively the lattices $D_8\oplus \langle -4\rangle,E_7\oplus A_1\oplus \langle -4\rangle,E_8\oplus \langle -6\rangle,E_8\oplus \langle -4\rangle, D_4^2\oplus \langle -4\rangle$. We claim that the K3 surfaces $X$ with $\NS(X)=U\oplus E_8\oplus L$, with $L$ one of these $5$ lattices have positive entropy. Consider $\NS(X)=U\oplus E_8\oplus E_8\oplus \langle -4\rangle$, as the others are analogous. Since the rank of $E_8\oplus E_8$ is $16>10$, by Watson's list \cite{nebe} it is not unique in its genus. This implies the existence of two distinct elliptic fibrations on $U\oplus E_8\oplus E_8$, say $E_1,E_2$. Then these two fibrations extend to elliptic fibrations $F_1=[E_1,0],F_2=[E_2,0]$ on $U\oplus E_8\oplus E_8\oplus \langle -4\rangle$ with infinitely many sections, as the orthogonal complements $F_1^\bot,F_2^\bot$ are not generated by roots (since both the orthogonal complements contain $\langle -4\rangle$ as a direct summand).

We were able to study most lattices in this fashion, and only $7$ genera of lattices remain. We switch back to the lattices $E_8\oplus L$. We can simply study with \texttt{Magma} the $2-$ or $3-$Neighbors of the lattices
$$E_8 \oplus A_1^9, E_8\oplus D_4\oplus A_1^5,E_8\oplus D_4^2\oplus A_1,E_8\oplus E_6\oplus A_2\oplus A_1,E_8\oplus E_7\oplus A_1^2,E_8\oplus D_7\oplus A_2,E_8\oplus L_4.$$
They all contain at least two non-isometric non-root-overlattices in the genus (cf. the ancillary file \linebreak[4]\texttt{Remaining7NonRootOverlattices}), concluding the proof.
\end{proof}

\begin{oss}
The same approach could be used to study K3 surfaces of smaller Picard rank. Indeed, \cite{nikulin9}, Corollary 1.12.3, shows that any trascendental lattice $\T(X)$ of rank $\le 6$ embeds into the unimodular lattice $U^2\oplus E_8$. Therefore, if $X$ is a K3 surface with $\rho(X)\ge 16$, its N\'eron-Severi lattice is isomorphic to $U\oplus E_8\oplus L$, for a certain negative definite lattice $L$. However, already in Picard rank $18$, we find lattices $L$ such that $E_8\oplus L$ admits a unique non-root-overlattice in the genus. Two examples are given by 
$$L=D_8,E_7\oplus A_1.$$
This corresponds to the fact that the K3 surfaces with $\NS(X)\cong U\oplus E_8\oplus D_8$ or $\NS(X)\cong U\oplus E_8\oplus E_7\oplus A_1$ admit a unique elliptic fibration with infinitely many sections \textit{up to automorphisms} ($N^{pos}(X)=1$ in the notation at the end of Section 3). This approach based on the study of the genus is thus not sufficient to decide whether these K3 surfaces have positive entropy.
\end{oss}

\nocite{*}
\printbibliography

\end{document}